\documentclass[11pt,a4paper]{article}
\usepackage[english]{babel}
\usepackage{amsmath,amsthm,amssymb,latexsym}
\usepackage[ruled,linesnumbered]{algorithm2e}
\usepackage[dvips]{graphicx}

\usepackage{lineno}
\usepackage{color}

\newtheorem{thm}{Theorem}[section]
\newtheorem{theorem}[thm]{Theorem}
\newtheorem{corollary}[thm]{Corollary}
\newtheorem{proposition}[thm]{Proposition}

\newtheorem{problem}[thm]{Problem}
\newtheorem{lemma}[thm]{Lemma}

\newenvironment{remark}
  {{\parindent=0pt {\bf Remark.}}}
   {\vspace{0.1cm}}

\begin{document}

\title{Shortcut sets for the locus of plane Euclidean networks}

\author{\normalfont\fontsize{12}{14}\selectfont
Jos\'e C\'aceres$^1$\and
Delia Garijo$^2$\and
Antonio Gonz\'alez$^2$\and
Alberto M\'arquez$^3$\and \\
Mar{\'i}a Luz Puertas$^1$\and \\
Paula Ribeiro$^4$}

\addtocounter{footnote}{1} \footnotetext{Departamento de Matem\'aticas, Universidad de Almer\'{\i}a, Spain. Email addresses:\texttt{\{jcaceres,mpuertas\}@ual.es}. Partially supported by project MTM2014-60127-P.}
\addtocounter{footnote}{1} \footnotetext{Departamento de Matem\'atica Aplicada I, Universidad de Sevilla, Spain. Email addresses: \texttt{\{dgarijo,gonzalezh\}@us.es}. Partially supported by project MTM2014-60127-P.}
\addtocounter{footnote}{1} \footnotetext{Departamento de Matem\'atica Aplicada I, Universidad de Sevilla, Spain. Email address: \texttt{almar@us.es}.}
\addtocounter{footnote}{1} \footnotetext{Instituto Superior de Engenharia, Universidade do Algarve, Faro, Portugal. Email address: \texttt{pribeiro@ualg.pt}.}
\maketitle

\begin{abstract}
We study the problem of augmenting the locus $\mathcal{N}_{\ell}$ of a plane Euclidean network $\mathcal{N}$ by inserting iteratively a finite set of segments, called \emph{shortcut set}, while reducing the diameter of the locus of the resulting network. There are two main differences with the classical augmentation problems: the endpoints of the segments are allowed to be points of $\mathcal{N}_{\ell}$ as well as points of the previously inserted segments (instead of only vertices of $\mathcal{N}$), and the notion of diameter is adapted to the fact that we deal with $\mathcal{N}_{\ell}$ instead of $\mathcal{N}$. This  increases enormously the hardness of the problem but also  its possible practical applications to network design. Among other results, we characterize the existence of shortcut sets, compute them in polynomial time, and analyze the role of the convex hull of $\mathcal{N}_{\ell}$ when inserting a shortcut set. Our main results prove that, while the problem of minimizing the size of a shortcut set is NP-hard, one can always determine in polynomial time whether inserting only one segment suffices to reduce the diameter.
\end{abstract}

\section{Introduction}

A \emph{geometric network} of points in the plane is an undirected graph whose vertices are points in $\mathbb{R}^2$ and whose edges are straight-line segments connecting pairs of points. When edges are assigned lengths equal to the Euclidean distance between their endpoints, the  geometric network is called \emph{Euclidean network}. These networks are the objects of study in this paper; concretely, we deal with \emph{plane} Euclidean networks in which there are no crossings between edges. For simplicity and when no confusion may arise, we shall simply say network, it being understood as plane Euclidean network and, unless otherwise stated, networks are assumed to be connected.

 It is natural to distinguish between a network $\mathcal{N}$ and its locus, denoted by  $\mathcal{N}_{\ell}$, where one is considering not only the vertices of  $\mathcal{N}$ but the set of all points of the Euclidean plane that are on $\mathcal{N}$ (thus, $\mathcal{N}_{\ell}$ is treated indistinctly as a network or as a closed point set). Roughly speaking, in order to compute the \emph{distance} between two vertices of $\mathcal{N}$ one has to sum the lengths of the edges of a shortest path connecting them, and the \emph{diameter} of $\mathcal{N}$ is the maximum among the distances between any two vertices. This concept is naturally extended to $\mathcal{N}_{\ell}$; the difference is that now the distances that must be computed are not only between vertices but also between two arbitrary points on edges of $\mathcal{N}$. In this setting, a shortest path connecting two points of  $\mathcal{N}_{\ell}$ may consist of a number of edges and one or two fragments of edges (in which the points to be connected are located). When computing the distance, we sum the lengths of all the edges in the path and the lengths of the fragments\footnote{The diameter of $\mathcal{N}_{\ell}$ is the \emph{generalized diameter} of $\mathcal{N}$, which was introduced in \cite{CG82} and called \emph{continuous diameter} in \cite{DCGMS15,HLN91}, but we use the context of \emph{locus} because it fits better to our purpose in this paper.}. See Figure 1.
 
\begin{figure}[!hbt]
  \centering
  \includegraphics[width=7cm]{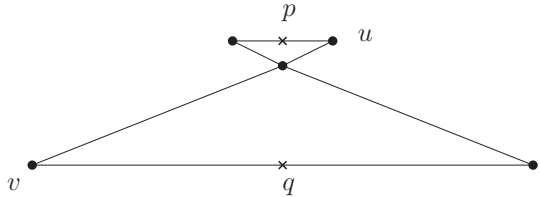}
  \caption{The distance between $u$ and $v$ is the diameter of the network, but the diameter of its locus is the distance between $p$ and $q$.}
\label{fig:delia1}
\end{figure}

In this paper we  study the following problem:

\begin{problem}\label{ourproblem} Given a plane Euclidean network $\mathcal{N}$, insert a finite set of segments $\mathcal{S}=\{s_1, \ldots, s_k\}$ in order to reduce (or minimize) the diameter of the locus of the resulting network, provided that the endpoints of segment $s_1$ are on  $\mathcal{N}_{\ell}$ and the endpoints of $s_i$, $2\leq i\leq k$, are on $\mathcal{N}_{\ell}\cup \{s_1,\ldots, s_{i-1}\}$. (We say that $\mathcal{S}$ is a \emph{shortcut set} for $\mathcal{N}_{\ell}$.)\end{problem}

This type of problems is mainly motivated by urban network design: to introduce shortcut sets is a way of improving the worst-case travel time along networks of roads in a city, highways, etc. Note that such models are considering the locus of the network, which is also used for related applications to location analysis \cite{BDCDGMSS13} and feed-link problems \cite{ABBJJKLLSS11,BDCDGMSS13,SS15}. One can find many other applications of Euclidean networks to robotics, telecommunications networks, computer networks, and flight scheduling. Any collection of objects that have some connections between them can be modeled as a geometric network. See \cite{ABCGHSV08,FPZW04,L03} for more details and references.

Problem \ref{ourproblem} also belongs to the class of \emph{graph augmentation} problems. Concretely, it is a variant of the Diameter-Optimal $k$-Augmentation Problem for edge-weighted geometric graphs, where one has to insert $k$ additional edges (i.e., a shortcut set of size $k$ with segments connecting vertices) to an edge-weighted plane geometric graph  in order to minimize the diameter of the resulting graph. Whilst there are numerous studies on the non-geometric version of this problem (see for instance \cite{EGR98,FGGM14,GGKSS2015}), when the input is an edge-weighted plane geometric graph, there are very few results  not only on this problem but on graph augmentation in general (see the survey \cite{HT13} for results and references on graph augmentation over plane geometric graphs). Farshi et al.~\cite{FGG05} designed approximation algorithms for the problem of adding one edge to a plane Euclidean network in $\mathbb{R}^d$ while minimizing the dilation. The same problem is considered in \cite{LW08,W10} but for networks embedded in a metric space. 

When considering the locus  $\mathcal{N}_{\ell}$ instead of the network $\mathcal{N}$ (setting of Problem \ref{ourproblem}), the hardness of the problem is enormously increased which motivates that, to the best of our knowledge, there are only two previous papers on this topic, both restricted to specific families of graphs. Yang \cite{yang} studied the restriction of Problem \ref{ourproblem} to inserting only one segment, called  \emph{shortcut}, to the locus of a simple polygonal chain\footnote{In fact, the restriction of Problem \ref{ourproblem} to inserting only one segment to the locus of a general plane Euclidean network was proposed by Cai \cite{C} as a personal communication to Yang \cite{yang}.}. Among other results, he designed three different approximation algorithms to compute an \emph{optimal} shortcut, i.e., a shortcut that minimizes the diameter among all shortcuts. On the other hand, in a very recent paper, De Carufel et al.~\cite{DCGMS15} consider the possibility of inserting more than one segment (all of them called shortcuts) to the locus of a Euclidean network $\mathcal{N}$ to minimize the diameter. They define a shortcut  as a segment with endpoints always on $\mathcal{N}_{\ell}$ and smaller length than the distance in $\mathcal{N}_{\ell}$ between its endpoints. Nevertheless,
 their sets of shortcuts do not necessarily decrease the diameter of the resulting network.
 There are more differences with our definition of shortcut set: their shortcuts only intersect $\mathcal{N}_{\ell}$ on its endpoints, and the possible intersection points between their shortcuts (when inserting more than one) are not considered as points of the resulting network. With this notion, they develop a study on paths and cycles determining in linear time optimal shortcuts for paths and  optimal pairs of shortcuts for convex cycles.

Our main contribution in this paper is to provide the first approach to Problem \ref{ourproblem} for general plane Euclidean networks.
Concretely, in Section~\ref{sec:upper}, we first characterize the networks $\mathcal{N}$ whose locus $\mathcal{N}_{\ell}$ admits a shortcut set. For such networks, we find in polynomial time a shortcut set for $\mathcal{N}_{\ell}$ giving an upper bound on its size. The section concludes by studying the connection between the diameter of the convex hull of $\mathcal{N}_{\ell}$ and the diameter of the resulting network when inserting a shortcut set to $\mathcal{N}_{\ell}$.


Section~\ref{sec:scn1} is devoted to prove that it is always possible to determine in polynomial time whether inserting only one segment to $\mathcal{N}_{\ell}$ suffices to reduce the diameter. Our method computes such a segment in case of existence, and combines a remarkable number of tools
that allow us to decide on the existence of that segment via analyzing intersections of certain arrangements of curves and certain regions. 

Section 4 contains our first approach to two hard problems: to minimize the size of shortcut sets and to minimize the diameter of the resulting networks when inserting shortcut sets. We consider shortcut sets for non-connected networks, which is essential for proving our main result in this section: the NP-completeness of the problem of deciding whether the minimum size of a shortcut set (for $\mathcal{N}_{\ell}$ not necessarily connected) is smaller or equal than a fixed natural number.

 We conclude the paper in Section~\ref{sec:CCRR} with some comments and open problems.

\section{Existence of shortcut sets}\label{sec:upper}

We begin with some notations and definitions. As usual, $V(\mathcal{N})$ and $E(\mathcal{N})$ denote, respectively, the vertex-set and the edge-set of a network $\mathcal{N}$, and $\delta(u)$ is the {\em degree} of $u\in V(\mathcal{N})$. An edge $uv\in E(\mathcal{N})$ is {\em pendant} if either $u$ or $v$ is a {\em pendant} vertex (i.e., has degree 1). We write $p\in \mathcal{N}_{\ell}$ for a point $p$ on $\mathcal{N}_{\ell}$, and distinguish the vertices of $\mathcal{N}$ (which are not assumed to be in general position) saying that $V(\mathcal{N})\subseteq \mathcal{N}_{\ell}$.

For $p,q\in \mathcal{N}_{\ell}$, a \emph{$p$-$q$ path} $P$ is a sequence $pu_1\dots u_kq$ such that $u_1u_2,\dots,u_{k-1}u_k\in E(\mathcal{N})$, $p$ is a point on an edge ($\neq u_1u_2$) incident with $u_1$, and $q$ is a point on an edge ($\neq u_{k-1}u_k$) incident with $u_k$. The \emph{length} of $P$, written as $|P|$, is the sum of the lengths of all edges $u_iu_{i+1}$ plus the lengths of the segments $pu_1$ and $qu_k$ (which are edges when $p,q\in V(\mathcal{N})$).
The {\em distance} $d(p,q)$ between points $p,q\in \mathcal{N}_{\ell}$ is the length of a shortest $p$-$q$ path in $\mathcal{N}_{\ell}$. The \emph{eccentricity} of a point $p\in \mathcal{N}_{\ell}$ is ecc$(p)=\max_{q \in \mathcal{N}_{\ell}} d(p,q)$ and the {\em  diameter} of $\mathcal{N}_{\ell}$ is $ {\rm diam}(\mathcal{N}_{\ell}) = \max_{p\in \mathcal{N}_{\ell}} {\rm ecc}(p)$.
Two points $p,q \in \mathcal{N}_{\ell}$ are {\em diametral} whenever $d(p,q)={\rm diam}(\mathcal{N}_{\ell})$; when in addition $p$ and $q$ are vertices, they are called {\em diametral vertices} of $\mathcal{N}_{\ell}$. These definitions for $\mathcal{N}$ (instead of $\mathcal{N}_{\ell}$) are analogous (and well-known), taking distances only between vertices of $\mathcal{N}$. When necessary, we shall use $d_H$, ecc$_H$ and $d_e$ for, respectively, distance on a network $H$, eccentricity on $H$, and Euclidean distance.

As we set in Problem \ref{ourproblem}, we define a {\em shortcut set} for the locus $\mathcal{N}_{\ell}$ of a network $\mathcal{N}$ as a finite set ${\cal S}=\{s_1,\dots, s_k\}$ of segments where $s_1$ has endpoints on $\mathcal{N}_{\ell}$ and $s_i$, $2\leq i\leq k$, has endpoints on $\mathcal{N}_{\ell}\cup \{s_1,\ldots, s_{i-1}\}$ satisfying that
$ {\rm diam}(\mathcal{N}_{\ell} \cup  {\cal S})<{\rm diam}(\mathcal{N}_{\ell})$; see Figure \ref{fig:delia2}(a). When ${\cal S}$ consists of only one segment $s$ we say that $s$ is a \emph{shortcut} (in this case, $\mathcal{N}_{\ell}  \cup {\cal S}$ is simply $\mathcal{N}_{\ell} \cup s$),  and a shortcut $s$ is called {\em simple} if its two endpoints are the only intersection points with $\mathcal{N}_{\ell}$; see Figure \ref{fig:delia2}(b). We want to point out that our definition of shortcut comes from that given in \cite{yang}, which includes as a possibility the equality of the diameters of $\mathcal{N}_{\ell}$ and $\mathcal{N}_{\ell}\cup s$ (and so the locus of every network has a shortcut), but we believe that the intuitive idea of shortcut is better captured by excluding that case. As it was explained in the Introduction, our definition also differs from the shortcuts of \cite{DCGMS15}.

\begin{figure}[ht]
\begin{center}
\begin{tabular}{cccccccc}
\includegraphics[width=0.25\textwidth]{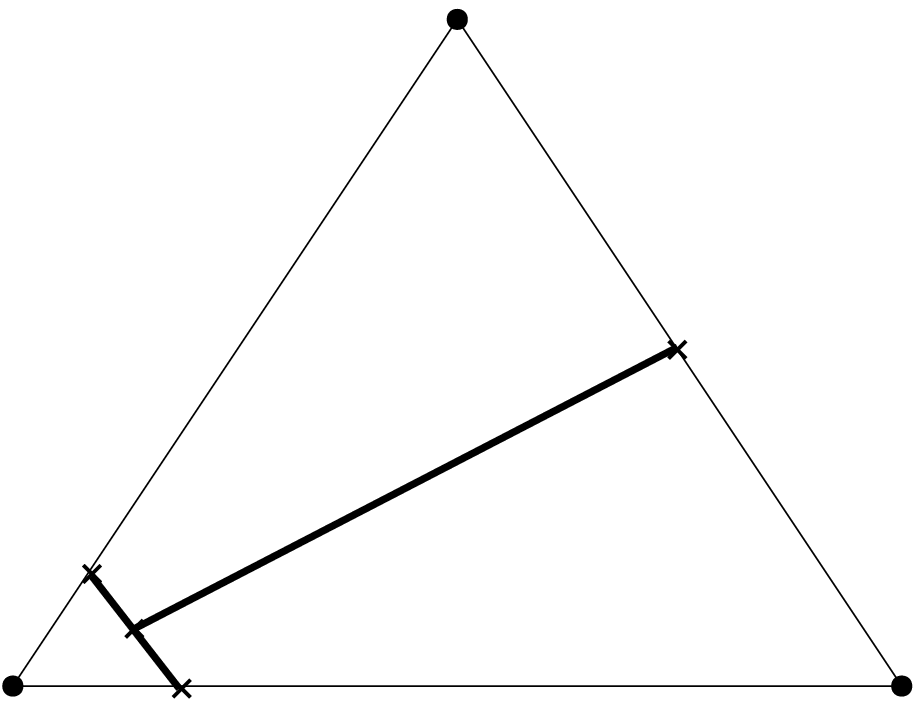}
&&&&&&&
\includegraphics[width=0.25\textwidth]{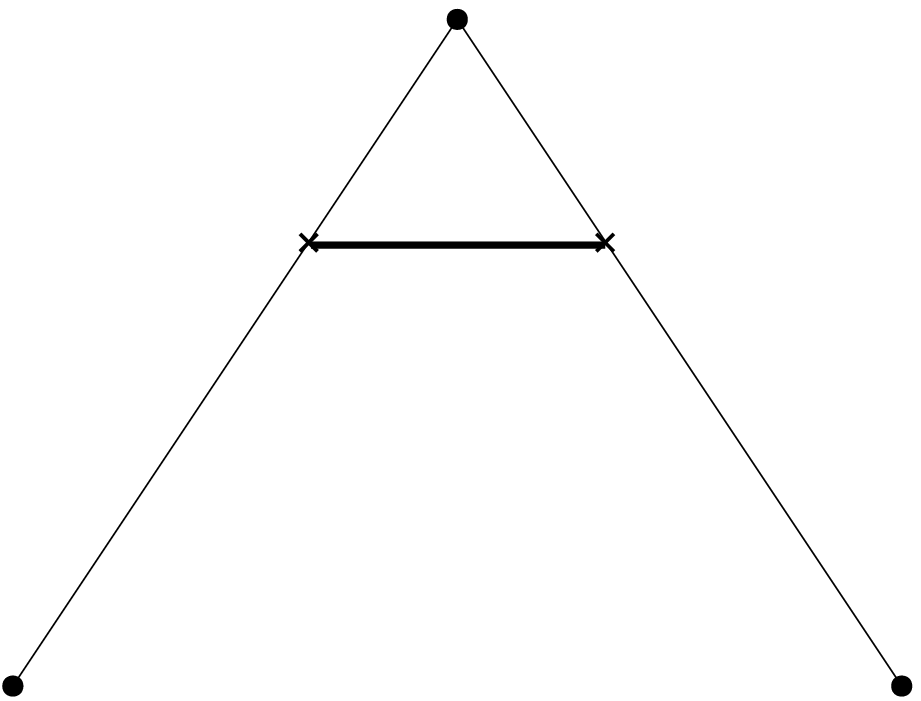}\\
(a)&&&&&&&(b)
\end{tabular}
\end{center}
\caption{Depicted as thick segments: (a) A shortcut set of size 2, (b) a simple shortcut.}\label{fig:delia2}
\end{figure}

One can easily find networks $\mathcal{N}$ whose locus $\mathcal{N}_{\ell}$ has no shortcut (a triangle, for example) and even no shortcut set of any size (a straight path). The characterization of the networks $\mathcal{N}$ whose locus $\mathcal{N}_{\ell}$ admits a shortcut set is given in Theorem  \ref{characterization} below. We use $CH(\mathcal{N}_{\ell})$ to denote the convex hull of $\mathcal{N}_{\ell}$; note that $CH(\mathcal{N}_{\ell})=CH(V(\mathcal{N}))$. Observe also that the distance between any two points in  $CH(\mathcal{N}_{\ell})$ is their Euclidean distance which easily leads to the fact that  ${\rm diam}(CH(\mathcal{N}_{\ell}))\leq {\rm diam}(\mathcal{N}_{\ell})$.

\begin{theorem}\label{characterization}
 Let $\mathcal{N}$ be a plane Euclidean network. Then, the following statements are equivalent:
\begin{enumerate}
\item[(i)] $\mathcal{N}_{\ell}$ admits a shortcut set.
\item[(ii)] The segment defined by any two diametral vertices is not contained in $\mathcal{N}_{\ell}$.
\item[(iii)] ${\rm diam}(CH(\mathcal{N}_{\ell}))<{\rm diam}(\mathcal{N}_{\ell})$.
\end{enumerate}
\end{theorem}

\begin{proof}
\noindent $\left( i \Longleftrightarrow ii \right)$ If $\mathcal{N}_{\ell}$ contains a segment defined by diametral vertices, say $u$ and $ v$, then $d_e(u,v)=d_{\mathcal{N}_{\ell}}(u,v)={\rm diam}(\mathcal{N}_{\ell})$ and so there is no shortcut set for $\mathcal{N}_{\ell}$. Otherwise, there exists $u\in V(\mathcal{N})$ such that $\delta(u)=r\geq 2$; let  $\{u_0, \ldots ,u_{r-1}\}$ be the set of its neighbours sorted clockwise.

Consider the oriented lines $m_i$ from $u$ to each $u_i$, and the corresponding right half-planes  $H_i^+$ (see Figure~\ref{fig:T1}).
For each $u_i$,  determine whether there exists a vertex $u_j$ such that $u_j\in H_i^+$ and $u_{j+1}\not\in H_i^+$ (where indices are taken modulo $r$). Observe that there must exist such a vertex $u_j$ for at least one of the $u_i$'s since $\mathcal{N}_{\ell}$ cannot be not a straight path, which contains a segment defined by two diametral vertices.

The angle $<u_iuu_j>$ is smaller than $\pi$ and so a segment $s_i$ crossing all edges $uu_k$, $i\leq k\leq j$, shortens all paths that contain any two of those edges  (see Figure~\ref{fig:T1}). Further, all segments $s_i$ must be placed very close to vertex $u$ (their lengths are as small as desired) so that ecc$(p)<{\rm diam}(\mathcal{N}_{\ell})$ for every $p\in \cup s_i$. Thus, the set of all those segments $s_i$ is a shortcut set.

\begin{figure}[!hbt]
  \centering
  \includegraphics[width=5cm]{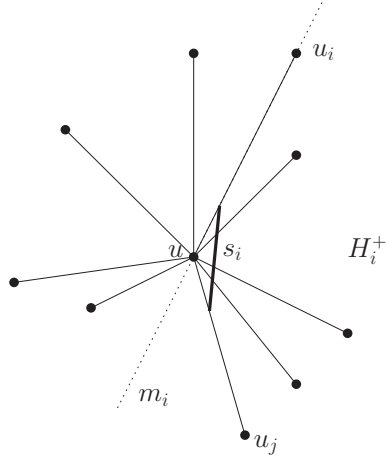}
  \caption{Segment $s_i$ (depicted as a thick segment) that shortens all paths containing any two of the intersected edges.}
  \label{fig:T1}
\end{figure}

\noindent $\left( ii \Longleftrightarrow iii \right)$ As it was said before,  ${\rm diam}(CH(\mathcal{N}_{\ell}))\leq {\rm diam}(\mathcal{N}_{\ell})$. If  ${\rm diam}(\mathcal{N}_{\ell})= {\rm diam}(CH(\mathcal{N}_{\ell}))$ then there exist $u,v\in V(\mathcal{N})$ such that $d_e(u,v)={\rm diam}(CH(\mathcal{N}_{\ell}))={\rm diam}(\mathcal{N}_{\ell})=d_{\mathcal{N}_{\ell}}(u,v)$, which implies that  segment $uv$ is contained in $\mathcal{N}_{\ell}$. The reverse implication can be proved analogously.
\end{proof}

The following corollary is a consequence of the preceding proof.

\begin{corollary}\label{bound}
Let $\mathcal{N}$ be a plane Euclidean network whose locus $\mathcal{N}_{\ell}$ admits a shortcut set. Then, it is always possible to compute in polynomial time a shortcut set for $\mathcal{N}_{\ell}$ of size at most $2|E(\mathcal{N})|-n_1$, where $n_1$ is the number of pendant vertices of $\mathcal{N}$.
\end{corollary}

\begin{proof}
The shortcut set for $\mathcal{N}_{\ell}$ constructed in the proof of Theorem \ref{characterization} (equivalence of statements (i) and (ii))  has
cardinality at most $$\sum_{u\in V(\mathcal{N}),\delta(u)\geq 2}\delta(u)  =  2|E(\mathcal{N})| - n_1,$$ and can obviously be constructed in polynomial time.
\end{proof}

Regarding the role of $CH(\mathcal{N}_{\ell})$ when inserting segments, one can easily check that a vertex $u\in V(\mathcal{N})$ (with $\delta(u)\geq 2$) on the border of $CH(\mathcal{N}_{\ell})$ requires only one segment for shortening all paths through $u$. This and similar considerations improve the upper bound of the preceding corollary, but we put the emphasis on its linearity with respect to $|V(\mathcal{N})|$. Nevertheless, the importance of $CH(\mathcal{N}_{\ell})$ goes much further as the following theorem reflects.

\begin{theorem}
Let $\mathcal{N}$ be a plane Euclidean network whose locus $\mathcal{N}_{\ell}$ satisfies that: $${\rm diam}(CH(\mathcal{N}_{\ell}))<{\rm diam}(\mathcal{N}_{\ell}).$$ Then,
 for every $\varepsilon>0$ such that ${\rm diam}(CH(\mathcal{N}_{\ell}))+\varepsilon<{\rm diam}(\mathcal{N}_{\ell})$
there exists a shortcut set ${\cal S}$ for $\mathcal{N}_{\ell}$ verifying that: $${\rm diam}(CH(\mathcal{N}_{\ell}))\leq {\rm diam}(\mathcal{N}_{\ell} \cup  {\cal S})<{\rm diam}(CH(\mathcal{N}_{\ell}))+\varepsilon.$$
\end{theorem}

\begin{proof}
Since ${\rm diam}(CH(\mathcal{N}_{\ell}))<{\rm diam}(\mathcal{N}_{\ell})$, by Theorem \ref{characterization}, $\mathcal{N}_{\ell}$ admits a shortcut set. Further, every shortcut set ${\cal S}$ for $\mathcal{N}_{\ell}$ verifies that
$CH(\mathcal{N}_{\ell})=CH(\mathcal{N}_{\ell} \cup  {\cal S})$ and so $${\rm diam}(CH(\mathcal{N}_{\ell}))={\rm diam}(CH(\mathcal{N}_{\ell} \cup  {\cal S}))\leq {\rm diam}(\mathcal{N}_{\ell} \cup  {\cal S}).$$ We next construct a shortcut set ${\cal S}$ that, in addition, satisfies that ${\rm diam}(\mathcal{N}_{\ell} \cup  {\cal S})<{\rm diam}(CH(\mathcal{N}_{\ell}))+\varepsilon$ for given $\varepsilon>0$ such that ${\rm diam}(CH(\mathcal{N}_{\ell}))+\varepsilon<{\rm diam}(\mathcal{N}_{\ell})$.

Consider the set $\mathcal{M}=\{p\in \mathcal{N}_{\ell} \, | \, {\rm ecc}(p)\geq M\}$ where $M={\rm diam}(CH(\mathcal{N}_{\ell}))+ \frac{\varepsilon}{4}$. This set is non-empty; otherwise  ${\rm ecc}(p)< M$ for every point $p\in \mathcal{N}_{\ell}$ and so ${\rm diam}(\mathcal{N}_{\ell})<M$, which contradicts ${\rm diam}(\mathcal{N}_{\ell})>{\rm diam}(CH(\mathcal{N}_{\ell}))+\varepsilon$. Note that this hypothesis is motivated by the fact that  ${\cal S}$ must be a non-empty set.

The set $\mathcal{M}$ is compact in $\mathbb{R}^2$. Indeed, $\mathcal{N}_{\ell}$ is compact in $\mathbb{R}^2$ (since it is a bounded closed point set) and so it suffices to prove that $\mathcal{M}$ is compact in $\mathcal{N}_{\ell}$.
Let $p\in \mathcal{N}_{\ell}\setminus \mathcal{M}$. Then ${\rm ecc}(p)<M$ and moreover, there is a neighborhood of $p$, say $\mathcal{N}_p\subseteq \mathcal{N}_{\ell}$, such that
 ${\rm ecc}(q)<M$ for every $q\in \mathcal{N}_p$. This implies that $\mathcal{N}_{\ell}\setminus\mathcal{M}$ is an open set in $\mathcal{N}_{\ell}$. Hence, $\mathcal{M}$ is a bounded closed set in $\mathcal{N}_{\ell}$ and so a compact.

The collection of balls $B(p,\frac{\varepsilon}{4})$ for $p\in \mathcal{M}$ is a cover of $\mathcal{M}$. Then, there is a finite subcover of $\mathcal{M}$, i.e., there exist $p_1, \ldots, p_k\in \mathcal{M}$ such that  \begin{equation}\mathcal{M}\subseteq \bigcup_{i=1}^k  B(p_i,\frac{\varepsilon}{4}).\end{equation}\label{rec1}

For $1\leq i\leq k$, let $\mathcal{M}_i=\{q\in \mathcal{N}_{\ell} \, | \, d(p_i,q)\geq M\}$. This set is non-empty because  ${\rm ecc}(p_i)\geq M$. Further it is bounded and, reasoning as above, one can check that it is a closed set in $\mathcal{N}_{\ell}$, and so a compact in $\mathbb{R}^2$. Since the collection of balls $B(q,\frac{\varepsilon}{4})$ for $q\in \mathcal{M}_i$ is a cover of $\mathcal{M}_i$ then there
 exist $q_1^i, \ldots, q_{r_i}^i\in \mathcal{M}_i$ such that
\begin{equation}\mathcal{M}_i\subseteq \bigcup_{t=1}^{r_i} B(q_t^i,\frac{\varepsilon}{4}).\end{equation}\label{rec2}

Let ${\cal S}$ be the set of segments with endpoints $p_i$ and $q_j^i$ for $1\leq i\leq k$ and $1\leq j\leq r_i$.
We now prove that $d(p,q)<{\rm diam}(CH(\mathcal{N}_{\ell}))+\varepsilon$ for every $p,q\in \mathcal{N}_{\ell}\cup {\cal S}$.

Clearly, if either $p$ or $q$ do not belong to $\mathcal{M}$ then $$d(p,q)<M={\rm diam}(CH(\mathcal{N}_{\ell}))+ \frac{\varepsilon}{4}<{\rm diam}(CH(\mathcal{N}_{\ell}))+\varepsilon. $$
Thus, we can assume that $p,q\in \mathcal{M}$. By (1), there exist $p_i,p_j\in \mathcal{M}$ such that $d(p,p_i)< \frac{\varepsilon}{4}$ and $d(q,p_j)<\frac{\varepsilon}{4}$.
If $p_j\notin \mathcal{M}_i$ then $d(p_i,p_j)<M$ and so $$d(p,q)\leq d(p,p_i)+d(p_i,p_j)+d(p_j,q)<{\rm diam}(CH(\mathcal{N}_{\ell}))+ \frac{3\varepsilon}{4}<{\rm diam}(CH(\mathcal{N}_{\ell}))+\varepsilon.$$ Otherwise, by (2), there exists  $q_h^i\in\mathcal{M}_i$  such that $d(p_j,q_h^i)<\frac{\varepsilon}{4}$. Further, there is a segment in ${\cal S}$ with endpoints $p_i, q_h^i$ which implies that $$d(p_i,q_h^i)\leq {\rm diam}(CH(\mathcal{N}_{\ell}))<M.$$ Therefore
 $d(p_i,p_j)\leq d(p_i, q_h^i)+d(q_h^i,p_j)<M+\frac{\varepsilon}{4}$, and we obtain
 $$d(p,q)\leq d(p,p_i)+d(p_i,p_j)+d(p_j,q)<M+ \frac{3\varepsilon}{4}={\rm diam}(CH(\mathcal{N}_{\ell}))+\varepsilon.$$
 Thus, we can conclude that $ {\rm diam}(\mathcal{N}_{\ell} \cup  {\cal S})<{\rm diam}(CH(\mathcal{N}_{\ell}))+\varepsilon$.
\end{proof}

\section{Computing shortcuts}\label{sec:scn1}

This section is devoted to prove that one can always determine in polynomial time whether  $\mathcal{N}_{\ell}$ admits a shortcut, and in that case, to compute it. For clarity, we split this result into two: Proposition \ref{theorem:algorithm_shortcut-simple} considers the case of simple shortcuts and Theorem \ref{theorem:algorithm_shortcut} states the result for all shortcuts. First, we prove that ${\rm diam}(\mathcal{N}_{\ell})$ can be computed in polynomial time (Lemma~\ref{lemma:algorithm-diameter}).
It is important to note that the computation of ${\rm diam}(\mathcal{N}_{\ell})$ is done in \cite{CG82} for general Euclidean networks but we include a version of that result (restricted to our plane networks) because its proof plays a fundamental role in the proofs of Proposition \ref{theorem:algorithm_shortcut-simple} and Theorem~\ref{theorem:algorithm_shortcut}. The following technical lemma will be useful in the proof of Lemma~\ref{lemma:algorithm-diameter} and also in
 the remainder of the paper.

\begin{lemma} \label{lemma:2shortest} Let $\mathcal{N}$ be a plane Euclidean network whose locus $\mathcal{N}_{\ell}$ has diametral points $p,q$ placed on two different non-pendant edges $uv$ and $u'v'$ of $\mathcal{N}$. Then, there exist two different shortest $p$-$q$ paths, say $P_1$ and $P_2$, such that either $u,u'\in P_1$ and $v,v'\in P_2$, or $u,v'\in P_1$ and $v,u'\in P_2$ (see Figure~\ref{fig:L1}).
\end{lemma}

\begin{proof}
Given shortest paths $P_u$ and $P_v$ connecting, respectively, $u,q$ and $v,q$,  the paths $pu\cup P_u$ and $pv\cup P_v$ are necessarily shortest $p$-$q$ paths when $p$ and $q$ are diametral points. Otherwise, say that $pv\cup P_v$ is not a shortest path connecting $p$ and $q$, then there would be a point $p'\in pv$ placed sufficiently close to $p$ such that $p'u\cup P_u$ is a shortest $p'$-$q$ path. Hence $d(p',q)>d(p,q)={\rm diam}(\mathcal{N}_{\ell})$, a contradiction.

Analogously, we can construct the paths $qu'\cup P_{u'}$ and $qv'\cup P_{v'}$ which are shortest $p$-$q$ paths, and  $P_1$ and $P_2$ can
be easily chosen among those four paths.
\end{proof}

\begin{figure}[!hbt]
  \centering
  \includegraphics[width=6.5cm]{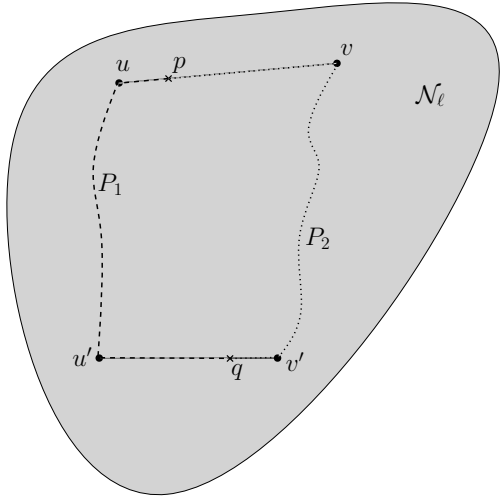}
  \caption{Paths $P_1$ and $P_2$ containing vertices $u,u'$ and $v,v'$, respectively.}
  \label{fig:L1}
\end{figure}


\begin{lemma} \label{lemma:algorithm-diameter}
Given a plane Euclidean network $\mathcal{N}$  with $n$ vertices, the diameter of its locus $\mathcal{N}_{\ell}$ can be computed in polynomial time in $n$.
\end{lemma}
\begin{proof}
We proceed as follows to find two diametral points on $\mathcal{N}_{\ell}$.

\begin{enumerate}

\item Compute the distances between any pair of vertices of $\mathcal{N}$.

\item For every pair of non-pendant edges $uv, u'v'\in E(\mathcal{N})$ compute: $$ \min  \left\lbrace \frac{d(u,v)+d(v,v')+d(v',u')+d(u',u)}{2}, \frac{d(u,v)+d(v,u')+d(u',v')+d(v',u)}{2} \right\rbrace$$ Here, we consider the case in which the diametral points are located on two non-pendant edges, and use the paths $P_1$ and $P_2$ of Lemma \ref{lemma:2shortest}. Thus, the preceding value is simply $(|P_1|+|P_2|)/2$.

\item For every pendant edge $uv\in E(\mathcal{N})$ with $\delta(u)=1$, and every non-pendant edge $u'v'\in E(\mathcal{N})$ compute: $$d(u,v)+\frac{d(v,u')+d(u',v')+d(v',v)}{2}$$

Note that if there is a diametral pair in which one of the points lies on a pendant edge, then it is the pendant vertex. Further, we can assume that the other point is located on a non-pendant edge; otherwise that pair is obtained in step (1).

\item Compute the maximum value among those obtained in the previous steps; this maximum is ${\rm diam}(\mathcal{N}_{\ell})$.

\end{enumerate}

 Since the values of steps (2) and (3) only depend on those of step (1), and they can be computed in polynomial time (see for instance \cite{CG82}), then we can also compute ${\rm diam}(\mathcal{N}_{\ell})$ in polynomial time.
\end{proof}

We are now ready to prove our main result in this section. As it was said before,  we split it into two results: Proposition  \ref{theorem:algorithm_shortcut-simple} below considers the case of simple shortcuts and its proof contains the main ideas of the proof of Theorem \ref{theorem:algorithm_shortcut} which is the general result for all shortcuts. In the proof of this general theorem, we shall mainly focus on the differences with the case of simple shortcuts.

\begin{proposition} \label{theorem:algorithm_shortcut-simple}
For every plane Euclidean network $\mathcal{N}$, it is always possible to determine in polynomial time whether there exists a simple shortcut for $\mathcal{N}_{\ell}$ and, in that case, to compute it.
\end{proposition}
\begin{proof}

There are two main steps in proving the result. First, we split the searching space into a polynomial number of ``equivalent'' regions (step (1)). Then, we design a method for seeking for a shortcut in each of those regions in polynomial time (step (2)).

\,

\noindent \emph{Step (1). Construction of a polynomial number of ``equivalent'' regions.}

\,

Consider two arbitrary vertical lines defining a strip enclosing $\mathcal{N}_{\ell}$. For each vertex $u \in V(\mathcal{N})$, take two horizontal segments defined by $u$ as one endpoint and the other in one of the vertical lines. Let ${\cal H}$ be the set of those $2n$ segments (formally, the direction of the segments of ${\cal H}$ must be different than those of the edges of $\mathcal{N}$ but, for simplicity, we can assume them to be horizontal).  We say that two lines are equivalent if they intersect the same segments of ${\cal H}$. Observe that there are $O(n^2)$ classes of equivalent lines. Indeed, for every pair of vertices $u,v\in V(\mathcal{N})$ consider the lines $m_{u^+v^+}$ and $m_{u^-v^-}$ parallel to segment $uv$ and leaving $u$ and $v$ in the same halfplane, and lines $m_{u^+v^-}$, $m_{u^-v^+}$ leaving $u,v$ in different halfplanes. These four lines must be placed sufficiently closed to $u$ and $v$; see Figure \ref{fig:th}(a). It is easy to check that every class of equivalent lines has a representative in the set $\{m_{u^+v^+}, m_{u^-v^-}, m_{u^+v^-}$, $m_{u^-v^+} \, | \, u,v\in V(\mathcal{N})\}$, whose cardinality is $O(n^2)$.

Given a line $ m$ that crosses two edges $e,e' \in E(\mathcal{N})$ and intersects no other edges in between them, let ${\cal P}_{e,e'}(m)$ be
the set of equivalent lines to $m$ that intersect edges $e$ and $e'$ (clearly, none of these lines intersect any edge in between $e$ and $e'$).
Observe that the intersection of a line of ${\cal P}_{e,e'}(m)$ with both edges determines a segment. Thus, $\mathcal P_{e,e'}(m)$ can also be viewed as the set of segments given by the corresponding intersections. It is well-known that the region of the plane defined by ${\cal P}_{e,e'}(m)$ has the shape of an hourglass, as illustrated in Figure~\ref{fig:th}(b); see for instance Section 3.1 of \cite{CG89}. Further, by the argument above, there are obviously $O(n^2)$ regions ${\cal P}_{e,e'}(m)$ for each two edges $e,e'\in E(\mathcal{N})$.

\begin{figure}[ht]
\begin{center}
\begin{tabular}{cccc}
\includegraphics[width=0.35\textwidth]{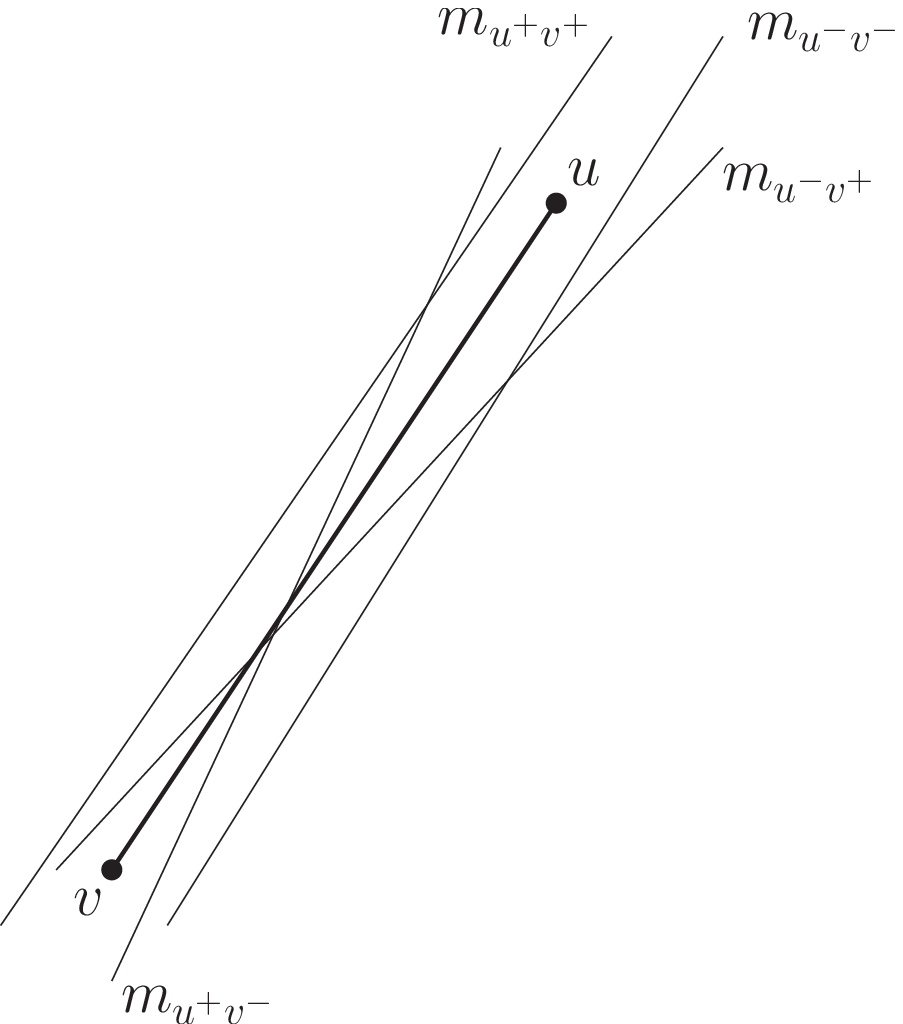}
&&&
\includegraphics[width=0.30\textwidth]{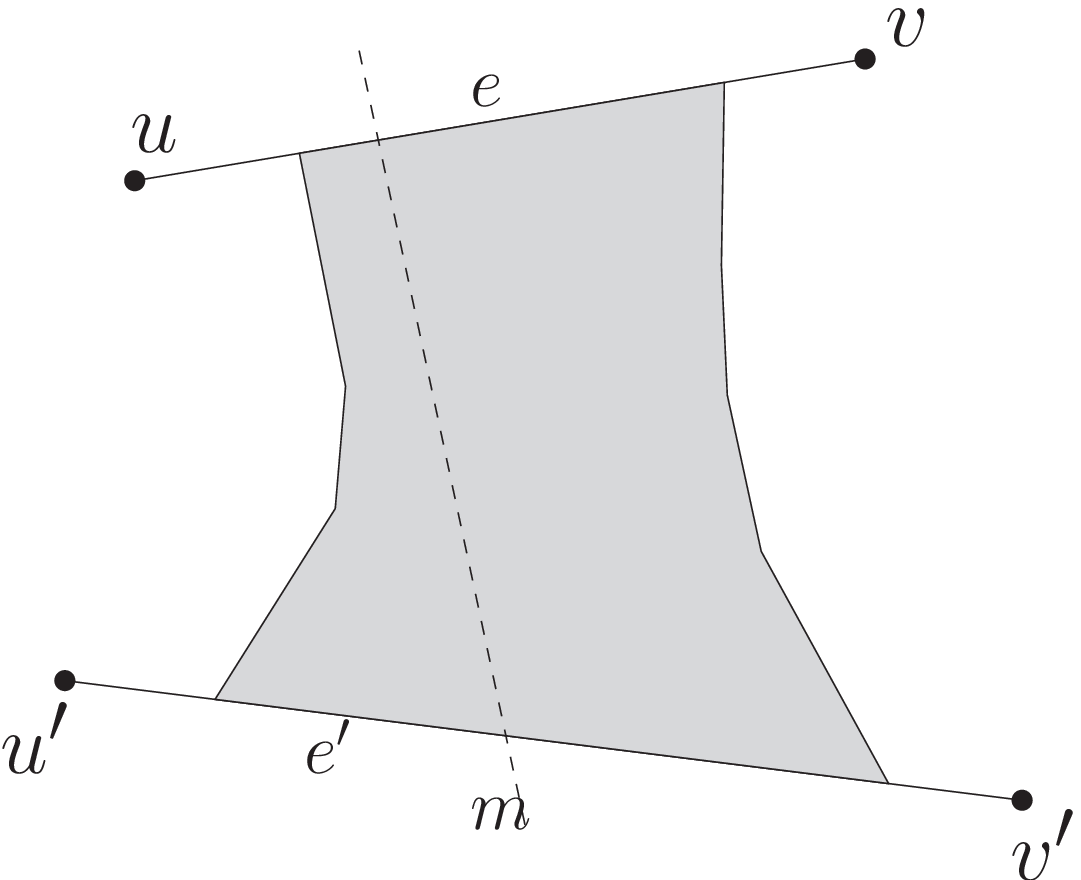}\\
(a)&&&(b)
\end{tabular}
\end{center}
\caption{(a) Lines associated to vertices $u$ and $v$, (b) region ${\cal P}_{e,e'}(m)$.}\label{fig:th}
\end{figure}

\,

\noindent \emph{Step (2).  Seeking for a simple shortcut in each region ${\cal P}_{e,e'}(m)$ in polynomial time.}

\,

Consider now a region ${\cal P}_{e,e'}(m)$. For simplicity, we shall say ${\cal P}(m)$ but one must recall that edges $e$ and $e'$ are associated to the region. We next show how to decide in polynomial time whether there is a segment $pp'\in {\cal P}(m)$ (endpoint $p$ on $e$ and $p'$ on $e'$) that is a shortcut for $\mathcal{N}_{\ell}$ (which would be a simple shortcut since lines of ${\cal P}(m)$ intersect no other edges in between $e$ and $e'$). This is: (i) $pp'$ must decrease the distance of all the diametral pairs of points on ${\cal N}_{\ell}$, and (ii)  ${\rm ecc}(q)<{\rm diam}(\mathcal{N}_{\ell})$ for every $q\in pp'$.

 Each segment of ${\cal P}(m)$ can be codified by its endpoints on $e$ and $e'$ as follows. If $e=uv$ and $e'=u'v'$ then a point $p$ on $e$ can be expressed as $p=ut+(1-t)v$, and a point on $e'$ is $p'=u't'+(1-t')v'$. Thus, a segment $pp'\in {\cal P}(m)$ is represented by a pair $(t,t')$.

The pairs $(t,t')$ must lie inside a region ${\cal R}(m)\subseteq \mathbb{R}^2$ whose boundary is given by the coordinates of the segments bounding ${\cal P}(m)$. Since region ${\cal P}(m)$ has the shape of an hourglass, the boundary of  ${\cal R}(m)$ is determined by two polygonal chains that are monotone with respect to the two axes,  one being increasing and the other decreasing.
Figure~\ref{fig:regionR}(b) shows an example in which points $(t_i,t'_i)$ and  $(s_i, s'_i)$ correspond, respectively, to the coordinates of the segments that form the left and right boundaries of the region ${\cal P}(m)$ in Figure~\ref{fig:regionR}(a). The polygonal chain given by points $(s_i, s'_i)$ is increasing ($s_i>s_j$ and $s'_i<s'_j$ for $i<j$)  and that given by points $(t_i,t'_i)$ is decreasing
($t_i<t_j$ and $t'_i>t'_j$ for $i<j$). Each point in ${\cal R}(m)$ represents a segment of ${\cal P}(m)$.

\begin{figure}[ht]
\begin{center}
\begin{tabular}{ccc}
\includegraphics[width=0.35\textwidth]{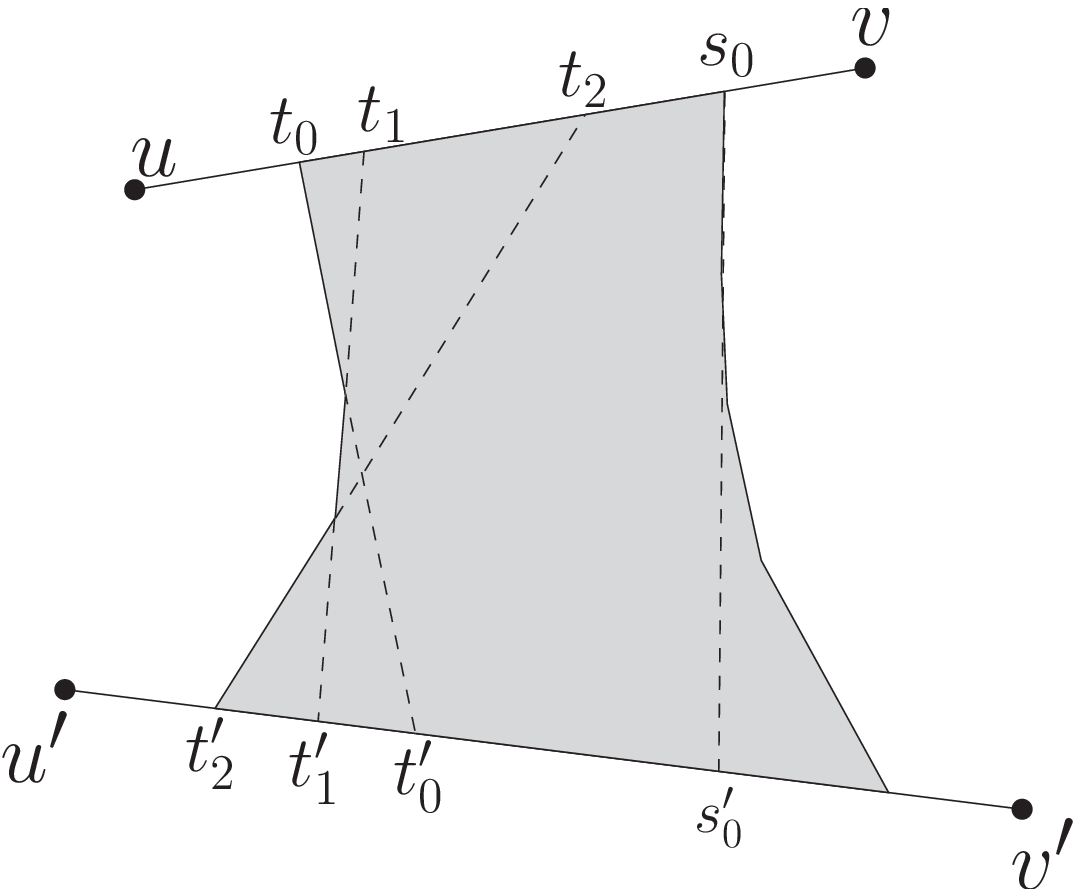}
&&
\includegraphics[width=0.50\textwidth]{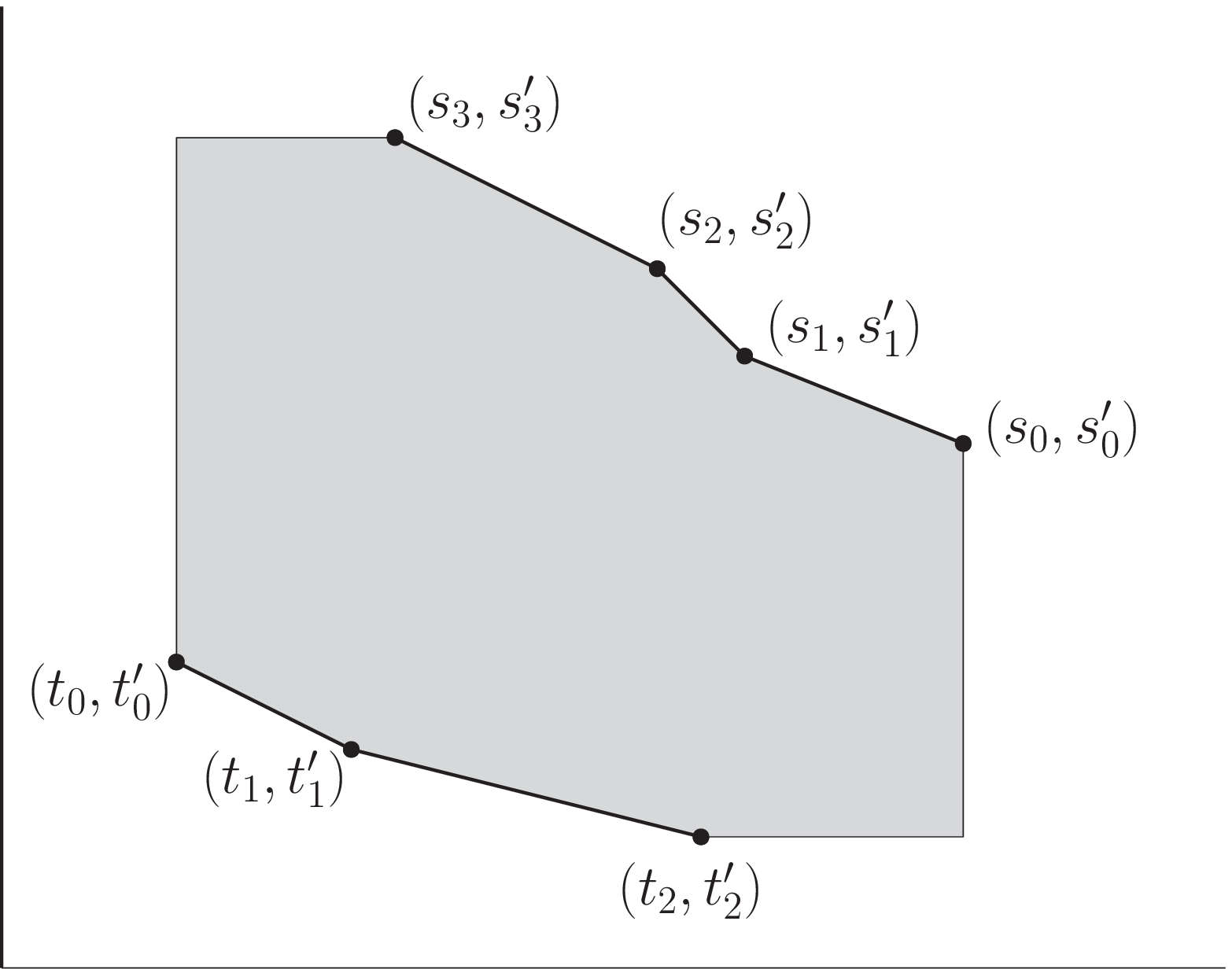}\\
(a)&&(b)
\end{tabular}
\end{center}
\caption{Segments in the boundary of region ${\cal P}(m)$ in (a) determine the boundary of region ${\cal R}(m)$ in (b).}\label{fig:regionR}
\end{figure}

By  Lemma~\ref{lemma:algorithm-diameter},  $d={\rm diam}(\mathcal{N}_{\ell})$ can be computed in polynomial time. Moreover, regarding condition (i) above, we know that the value $d$ is obtained by checking whether there are diametral points on $\mathcal{N}_{\ell}$ that may be: (a) two vertices,  (b) two points on different non-pendant edges, and (c) a pendant vertex and a point on a non-pendant edge. Thus, from Lemma 9 of \cite{BDCDGMSS13}, one can deduce that there is at most a quadratic set of diametral points.

Consider now the endpoints $u$ and $ u'$ of edges $e$ and $e'$, respectively. We distinguish three cases depending on the type of diametral points.

\,

\noindent \emph{Case (a).} Let $w,z\in V(\mathcal{N})$ be diametral vertices of $\mathcal{N}_{\ell}$. Suppose that $d_1=d(w,u)+d(z,u') \leq d(w,u')+d(z,u)$; otherwise we would follow the same argument by taking $d_1=d(w,u')+d(z,u)$. If there is a segment $pp'\in {\cal P}(m)$  in a path passing through $u$ and $u'$ that decreases $d(w,z)=d$, then $d_1+d(u,p)+d(p,p')+d(p',u') < d$ which is
\begin{equation}\label{parabola1}
 d(u,p)+d(p,p')+d(p',u')<d-d_1
\end{equation}

\noindent \emph{Case (b).} Let $w$ and $z$ be two diametral points located on two non-pendant edges $ab$ and $a'b'$, respectively. By Lemma \ref{lemma:2shortest}, one can consider two shortest $w$-$z$ paths $P_1$ and $P_2$ such that $d={\rm diam}(\mathcal{N}_{\ell})=(|P_1|+|P_2|)/2$. Assume that $a,a'\in P_1$ and $b,b'\in P_2$ (a similar argument would be used for $a,b'\in P_1$ and $b,a'\in P_2$).
If a path passing through $u,u'$ and containing a segment $pp'\in {\cal P}(m)$ decreases the value $(|P_1|+|P_2|)/2$ then either $d(a,a')$ or $d(b,b')$ is decreased, say $d(a,a')$ (analogous for $d(b,b')$). Thus,
$$d(b,a)+d(a,u)+d(u,p)+d(p,p')+d(p',u')+d(u'a')+d(a',b')+d(b',b) < 2d$$
Let $d_2=d(b,a)+d(a,u)+d(u'a')+d(a',b')+d(b',b)$ where we assume that  $d(a,u)+d(a',u') \leq d(a,u')+d(a',u)$ (otherwise we take this last value for $d_2$). Hence,
\begin{equation}\label{parabola2}
d(u,p)+d(p,p')+d(p',u')<2d-d_2
\end{equation}

\noindent \emph{Case (c).} Let $w,z\in \mathcal{N}_{\ell}$ be diametral points such that $w$ is the pendant vertex of a pendant edge $ww'$, and $z$ is a point on a non-pendant edge $ab$. This case is a combination of cases (a) and (b); by doing similar assumptions,  we obtain
\begin{equation}\label{parabola3}
d(u,p)+d(p,p')+d(p',u')<2d-d_3
\end{equation}
where $d_3=2d(w,w')+d(w',u)+d(u',a)+d(a,b)+d(b,w')$.

\,

Observe that for a fixed pair $(w,z)$, Inequations (\ref{parabola1}), (\ref{parabola2}), and (\ref{parabola3}) give us the interior points $(t,t')$ (corresponding to points $p,p'$) of certain conics $Q_{u,u'}^{w,z}$. They are conics because $d(u,p)$ and $d(u',p')$ are linear functions of $t$ and $t'$, $$d(u,p)=(1-t)d(u,v) \quad {\rm and} \quad d(u',p')=(1-t')d(u,v').$$
We can analogously construct the corresponding conics $Q_{u,v'}^{w,z}$, $Q_{v,u'}^{w,z}$, and $Q_{v,v'}^{w,z}$ for paths passing through $u,v'$ or $v,u'$ or $v,v'$ (instead of $u,u'$). When considering all types of diametral pairs $(w,z)$ we obtain an arrangement $\mathcal{Q}$ of conics.

Suppose now that there is a cell $\mathcal{C}$ in $\mathcal{Q} \cap  {\cal R}(m)$ that is contained, per each diametral pair, in at least one of its four associated conics. Then, all points in   $\mathcal{C}$ would represent
segments $pp'\in {\cal P}(m)$ that decrease the distance of all the diametral pairs of points on $\mathcal{N}_{\ell}$. Nevertheless, this would not imply that those segments $pp'$ are shortcuts since it must be guaranteed that the maximum eccentricity of the points on $pp'$ is smaller than $d$ (condition (ii) above). To do this, we construct another arrangement of conics $\mathcal{Q}'$ as follows.

Let $ab$ be a non-pendant edge of ${\cal N}$. Suppose that there is a segment $pp'\in {\cal P}(m)$ such that $\mathcal{N}_{\ell}\cup pp'$ has two diametral points placed, respectively, on $ab$ and on segment $pp'$. Take the two paths $P_1$ and $P_2$ of Lemma \ref{lemma:2shortest} and, without loss of generality, suppose that $a,p\in P_1$ and $b,p'\in P_2$. This implies that either $u$ or $v$ are on $P_1$ and either $u'$ or $v'$ are on $P_2$. Assume that $u\in P_1$ and $u'\in P_2$. The argument is analogous for the remaining cases.
Since the maximum eccentricity of the points on $pp'$ must be smaller than $d$ then
$$\frac{|P_1|+|P_2|}{2}<d \Rightarrow d(u,p)+d(p,p')+d(p',u')<2d-d_4$$
where $d_4=d(b,a)+d(a,u)+d(u',b)$. Similarly, given a non-pendant vertex $r\in V(\mathcal{N})$, if $r$ is diametral to some point on $pp'$ then
$$d(u,p)+d(p,p')+d(p',u')<2d-d_5$$ where $d_5=d(r,u)+d(r,u')$ (assuming that $r,u\in P_1$ and $r,u'\in P_2$).
It is easy to check that for pendant vertices, we obtain similar equations to the  previous ones. Thus, condition (ii) above is captured again by certain conics which give rise to the arrangement $\mathcal{Q}'$. Every vertex and every edge of $\mathcal{N}$ has an associated conic of $\mathcal{Q}'$. The interior points of a conic associated to, say a non-pendant edge, correspond to segments with maximum eccentricity with respect to that edge smaller than $d$.

We can conclude that there is a segment of ${\cal P}(m)$ that is a simple shortcut for $\mathcal{N}_{\ell}$ if and only if there exists a cell in $\mathcal{Q} \cap \mathcal{Q}' \cap {\cal R}(m)$ that is contained in all conics of  $\mathcal{Q}'$ and, per each diametral pair, in at least one of its four associated conics of $\mathcal{Q}$. Any point in such a cell represents a segment that is a simple shortcut for $\mathcal{N}_{\ell}$.

\,

\noindent \emph{Complexity.} As it was mentioned before, the set of diametral pairs of points on $\mathcal{N}_{\ell}$ is at most quadratic. Further, the sets of vertices and edges are linear and so the set of values needed to obtain our arrangements $\mathcal{Q}$ and $\mathcal{Q}'$ is polynomial. Those values depend on distances between vertices of $\mathcal{N}$, and so their computation can be done in polynomial time \cite{CG82} as well as the computation of the arrangements $\mathcal{Q}$ and $\mathcal{Q}'$ \cite{EGPPSS92}. Since ${\cal R}(m)$ is bounded by two monotone chains,  $\mathcal{Q} \cap \mathcal{Q}' \cap {\cal R}(m)$ can also be obtained in polynomial time.
The analysis of step (2) must be performed at most in the $O(n^2)$ regions ${\cal P}_{e,e'}(m)$ for each of the $O(n^2)$ pairs of edges $e,e'$ (regions obtained in step (1)).  Thus, the process is polynomial and the result follows.
\end{proof}

The following theorem is the general version of the preceding proposition; it includes non-simple shortcuts.

\begin{theorem} \label{theorem:algorithm_shortcut}
For every plane Euclidean network $\mathcal{N}$, it is possible to determine in polynomial time whether there exists a shortcut for $\mathcal{N}_{\ell}$ and, in that case, to compute it.
\end{theorem}
\begin{proof}
 The argument is essentially the same as in the proof of Proposition \ref{theorem:algorithm_shortcut-simple} although there are some important differences on whose description we focus next. Throughout this proof, steps (1) and (2) refer to the corresponding steps in the proof of  Proposition \ref{theorem:algorithm_shortcut-simple}.

 Since the shortcut may be non-simple, when constructing the regions ${\cal P}_{e,e'}(m)$ in step (1), we cannot assume that line $m$ intersects no other edges in between $e$ and $e'$ but it may intersect $k$ edges in between them, say a set $I=\{e_0=e, e_1, \ldots , e_{k+1}=e'\}$ of edges with $e_i=u_iu'_i$ (counting $e$ and $e'$). This implies that, for a candidate segment $pp'\in {\cal P}_{e,e'}(m)$ to be a shortcut, different diametral pairs $(w,z)$ and $(w',z')$ might use different fragments of segment $pp'$ in order to decrease $d={\rm diam}(\mathcal{N}_{\ell})$. For example,  one can have a $w$-$z$ path decreasing $d=d(w,z)$ that connects $w$ with $u_i$, uses the fragment of $pp'$ in between $e_i$ and $e_j$ and then connects $u'_j$ with $z$. There could also be a $w'$-$z'$ path decreasing $d=d(w',z')$ using another fragment of $pp'$. When trying to mimic the construction of the arrangements $\mathcal{Q}$ and $\mathcal{Q}'$ in step (2), this fact  generates the main differences with the proof of Proposition \ref{theorem:algorithm_shortcut-simple}. Nevertheless, we still have the $O(n^2)$ regions ${\cal P}_{e,e'}(m)$ for each two edges $e,e'\in E(\mathcal{N})$. For short, we again use ${\cal P}(m)$ instead of ${\cal P}_{e,e'}(m)$. Also, for each region ${\cal P}(m)$, a region ${\cal R}(m)\subseteq \mathbb{R}^2$ can be constructed as in step (2), and a segment $pp'\in {\cal P}(m)$ can be represented by a point $(t,t')\in {\cal R}(m)$.

For constructing arrangement $\mathcal{Q}$ in step (2), we consider the four combinations of endpoints of $e$ and $e'$ (i.e., paths passing through $u,u'$ or $u,v'$ or $v,u'$ or $v,v'$) but here
 we must take the $2(k+2)(k+1)$ suitable combinations of endpoints of different pairs of edges of set $I$ (still a polynomial number).  This means that, for each pair of endpoints of those pairs of edges, we analyze (as it was done in step (2)) the inequations obtained for a candidate segment $pp'\in {\cal P}(m)$ to be a shortcut which, in this general case, are of the form:
 $$d(u_i,p_i)+d(p_i,p'_j)+d(p'_j,u_j)<c$$
 where we have taken as an example (in order to show the type of inequation) endpoints $u_i$ of $e_i$ and $u_j$ of $e_j$. Further,
 $p_i$ and $p'_j$ are the intersection points of segment $pp'$ with edges $e_i$ and $e_j$, respectively; value $c$ is a constant that depends on $d={\rm diam}(\mathcal{N}_{\ell})$ as well as on several distances obtained by reasoning as in step (2) when considering the different types of diametral pairs.

The above inequations do not necessarily describe conics (as in step (2)) because the terms $d(u_i,p_i)$ and $d(p'_j,u_j)$ may not be linear functions of $t$ and $t'$, respectively. Clearly, the same happens for arrangement $\mathcal{Q}'$ since the inequations obtained for this arrangement are of the same type.
 Nevertheless, we shall show that the situation can be handled in a similar fashion of Proposition \ref{theorem:algorithm_shortcut-simple} by proving, as follows, that those loci of points are convex sets.

Let $(w,z)$ be a diametral pair of points. Suppose that  $d(w,z)=d$ is decreased by using a path $T_1$ passing through vertices $u_i$ and $u_j$ and containing the fragment of a segment $p_1p'_1\in{\cal P}(m)$ in between $e_i$ and $e_j$; this segment is represented by a point $(t_1,t_1')\in {\cal R}(m)$. Assume also that $d(w,z)=d$ is decreased by another path $T_2$ passing through the same vertices $u_i$ and $u_j$ and containing the fragment of a segment $p_2p'_2\in{\cal P}(m)$ in between $e_i$ and $e_j$; segment $p_2p'_2$ is represented by a point $(t_2,t_2')\in {\cal R}(m)$. To prove that the inequations corresponding to arrangement $Q$ describe convex sets, it suffices to show that every point $(t,t')$ on the segment with endpoints $(t_1,t_1')$ and $(t_2,t_2')$  represents a segment $pp'\in {\cal P}(m)$ such that  $d(w,z)=d$ is decreased by using a path, say $T$, passing through vertices $u_i$ and $u_j$ and containing the fragment of $pp'$ in between $e_i$ and $e_j$. To do this, we distinguish two cases.

\,

\noindent \emph{Case (a).} Vertices $u_i$ and $u_j$ are located at the same side of $p_1p'_1$ and $p_2p'_2$; see Figure \ref{fig:scg1}.

The result is straightforward when either $p_1p'_1$ and $p_2p'_2$ do not intersect or they do but on a point located outside the region defined by $e_i$ and $e_j$; see Figure \ref{fig:scg1}(a). Otherwise, all points on the segment with endpoints $(t_1,t_1')$ and $(t_2,t_2')$ represent segments in ${\cal P}(m)$ that pass through the intersection point of $p_1p'_1$ and $p_2p'_2$; see Figure \ref{fig:scg1}(b). Suppose that this intersection point is under the bisector defined by $e_i$ and $e_j$ as Figure \ref{fig:scg1}(b) shows. By comparing triangles, one can check that a path $T$ that uses the fragment of segment $pp'$ in between $e_i$ and $e_j$ is shorter than $T_2$ (similar for the intersection point over the bisector but using path $T_1$). Hence, $T$ decreases $d(w,z)=d$ since $T_2$ does.

\begin{figure}[ht]
\begin{center}
\begin{tabular}{cc}
\includegraphics[width=0.48\textwidth]{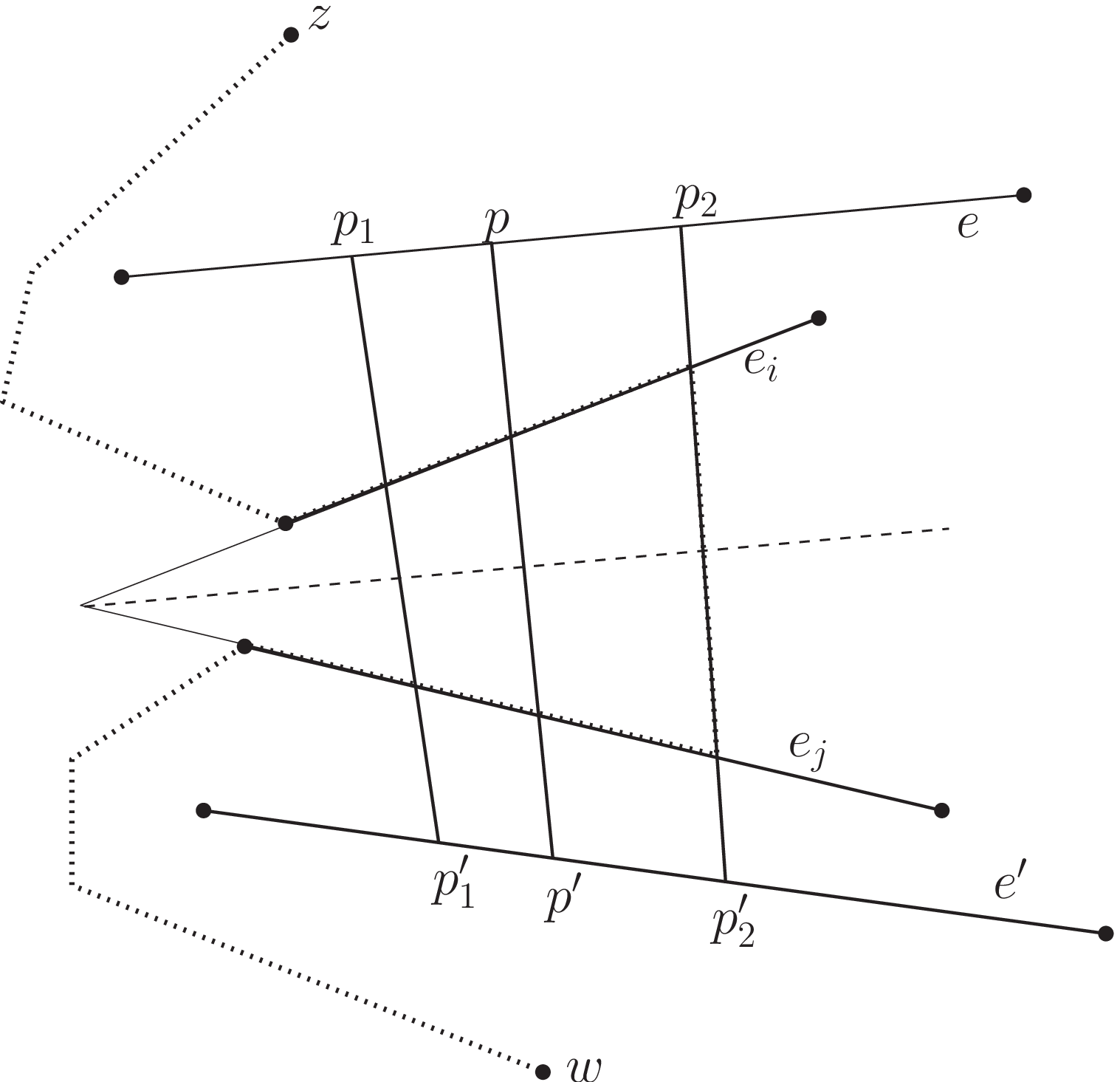}
&
\includegraphics[width=0.48\textwidth]{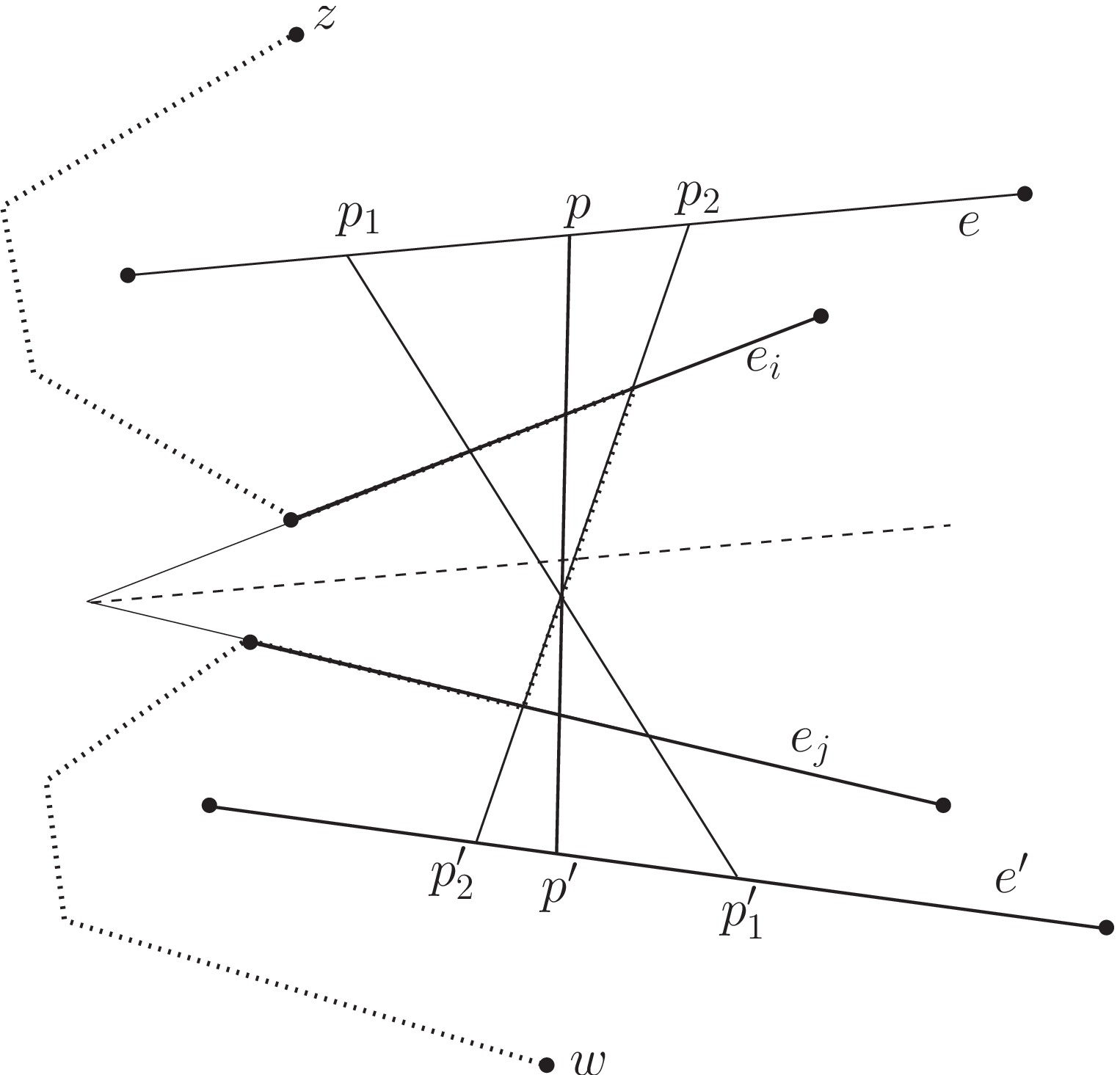}\\
(a)&(b)
\end{tabular}
\caption{(a) Segments $p_1p'_1$ and $p_2p'_2$ do not intersect, (b) segments $p_1p'_1$ and $p_2p'_2$ intersect on a point located in the region defined by $e_i$ and $e_j$.}\label{fig:scg1}
\end{center}
\end{figure}

\,

\noindent \emph{Case (b).} Vertices $u_i$ and $u_j$ are located at different sides of $p_1p'_1$ and $p_2p'_2$; see Figure \ref{fig:scg3}.

\begin{figure}[ht]
\begin{center}
\begin{tabular}{cc}
\includegraphics[width=0.48\textwidth]{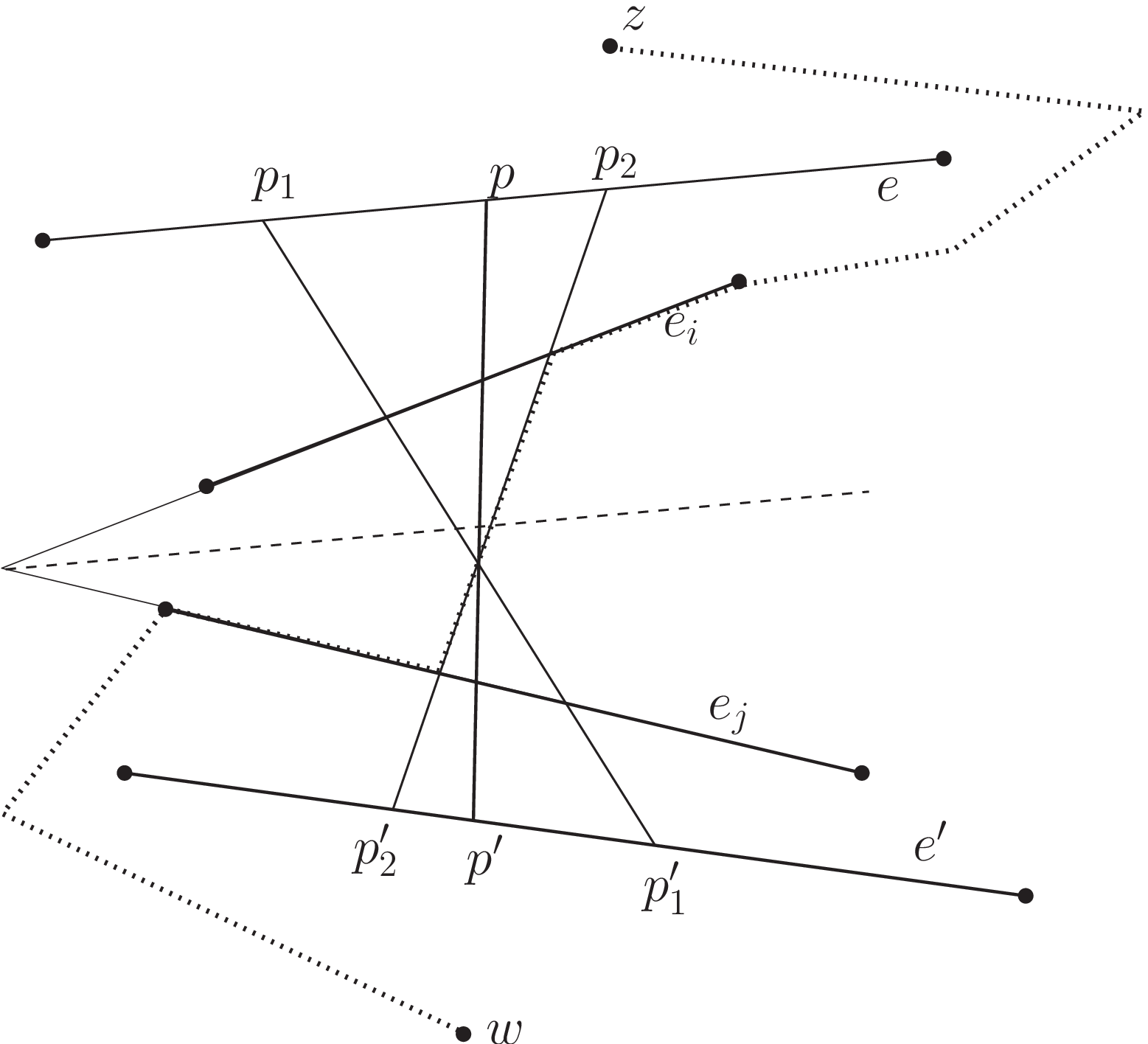}
&
\includegraphics[width=0.48\textwidth]{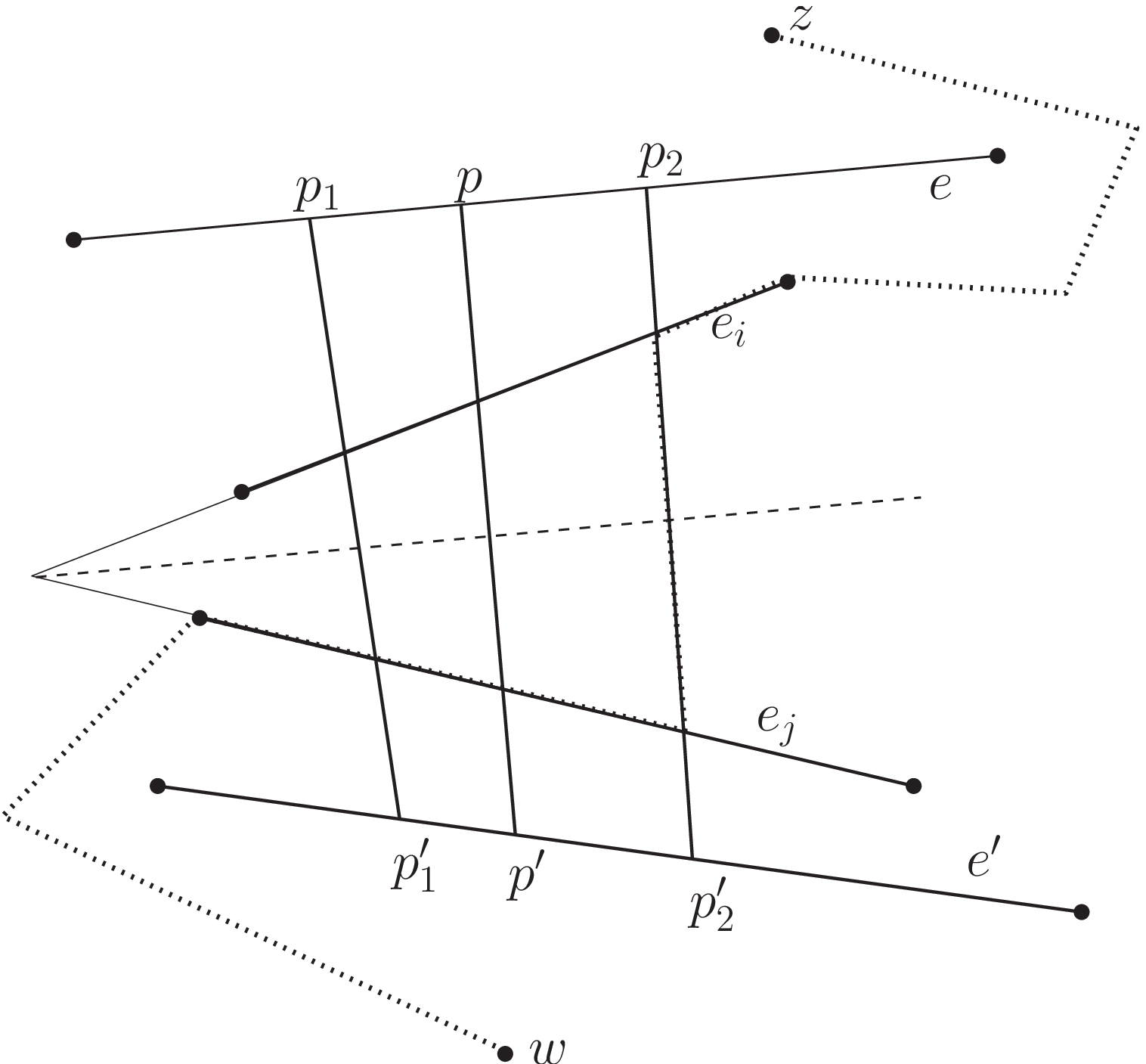}\\
(a)&(b)
\end{tabular}
\caption{Segments $p_1p'_1$ and $p_2p'_2$ intersect in (a) and do not intersect in (b).}\label{fig:scg3}
\end{center}
\end{figure}

If  $p_1p'_1$ and $p_2p'_2$ intersect then the result is straightforward since points $(t,t')$ represent segments that pass through that intersection point; see Figure \ref{fig:scg3}(a). Otherwise, consider a point  $(t,t')$  on the segment with endpoints  $(t_1,t_1')$ and $(t_2,t_2')$, i.e., $(t,t')=\lambda (t_1,t_1')+(1-\lambda)(t_2,t_2')$. 

Suppose first that  $\lambda=1/2$. A well-known property of a quadrilateral is that if the midpoints of opposite edges are connected (points $a'$ and $d'$ in Figure \ref{fig:quad}), then the length of the path $ba'd'd$ is smaller or equal than the length of at least one of the two paths in the quadrilateral connecting $b$ with $d$. Therefore, path $T$ is shorter than $T_1$ or $T_2$ and so it also decreases $d(w,z)=d$. For $\lambda\neq 1/2$, we can apply successively the same property, starting with either the quadrilateral with vertices $a,a',d,d'$ or  that with vertices $a',b,d'c$.

\begin{figure}[ht]
\begin{center}
\includegraphics[width=0.35\textwidth]{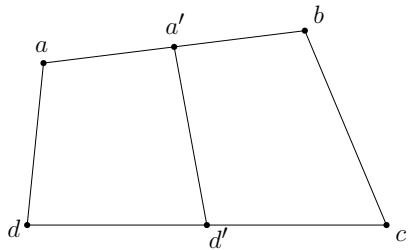}
\caption{A quadrilateral with vertices $a,b,c,d$ in which the midpoints $a'$ and $d'$ of, respectively, edges $ab$ and $cd$ are connected.}\label{fig:quad}
\end{center}
\end{figure}

Thus, we can conclude that the inequations corresponding to arrangement $\mathcal{Q}$ describe convex sets. One can analogously prove that arrangement $\mathcal{Q}'$ consists also of convex sets. Since the process is very similar, we omit the details and only specify the convexity condition in this case:  consider a segment $p_1p'_1\in{\cal P}(m)$ (which is a point $(t_1,t_1')\in {\cal R}(m)$) and suppose that the maximum eccentricity of the points on $p_1p'_1$ (in the network $\mathcal{N}_{\ell}\cup p_1p'_1$) is smaller than $d$. Further, as it was done in step (2),  assume that $u_i\in P_1$ and $u_j\in P_2$ for the corresponding paths $P_1$ and $P_2$ of Lemma \ref{lemma:2shortest} that give that maximum eccentricity. The same assumptions are done for another segment $p_2p'_2\in{\cal P}(m)$ (which is a point $(t_2,t'_2)\in {\cal R}(m)$). Thus, the convexity condition for $\mathcal{Q}'$ is that every point $(t,t')$ on the segment with endpoints $(t_1,t_1')$ and $(t_2,t_2')$  represents a segment $pp'\in {\cal P}(m)$ such that the maximum eccentricity of the points on $pp'$ (in the network $\mathcal{N}_{\ell}\cup pp'$) is smaller than $d$, assuming also that $u_i\in P_1$ and $u_j\in P_2$ for the corresponding paths $P_1$ and $P_2$ of Lemma \ref{lemma:2shortest} that give the maximum eccentricity.

Regarding the complexity, the only difference with the proof of Proposition \ref{theorem:algorithm_shortcut-simple} is that, instead of analyzing the intersection of two arrangements of conics with region ${\cal R}(m)$, we compute the intersection of two arrangements of convex sets with that region, which can also be done in polynomial time \cite{EGPPSS92} and so the result follows.
\end{proof}

\section{Shortcut number}\label{sec:families}

There are two natural, interesting but very hard questions related with shortcut sets: which is the minimum number of segments that have to be inserted to $\mathcal{N}_{\ell}$ in order to reduce the diameter? Among all shortcut sets of minimum size, how can we find the one that minimizes the diameter? It seems difficult to give an answer for general plane Euclidean networks and, of course, new techniques must be developed. In this section we briefly describe our results on these questions, among which highlights the NP-completeness of the problem of deciding whether the minimum size of a shortcut set is smaller or equal than a fixed natural number.

Let $\mathcal{N}$ be a  network such that $\mathcal{N}_{\ell}$ admits a shortcut set (according to Theorem \ref{characterization}). We define the {\em shortcut number}  of $\mathcal{N}_{\ell}$, denoted by  ${\rm scn}(\mathcal{N}_{\ell})$,
as the minimum size of a shortcut set. A shortcut set is called {\em optimal} if its size equals ${\rm scn} (\mathcal{N}_{\ell})$ and it minimizes
${\rm diam}(\mathcal{N}_{\ell} \cup {\cal S}) $ among all shortcut sets ${\cal S}$ of size ${\rm scn}(\mathcal{N}_{\ell})$. This notion of optimal shortcut set includes the optimal shortcut of Yang \cite{yang}, and has the same spirit as the optimal set of shortcuts of De Carufel et al.~\cite{DCGMS15}.

\begin{proposition} Let $\mathcal{N}$ be a plane Euclidean network whose locus $\mathcal{N}_{\ell}$ admits a shortcut set. Then, the following statements hold.
\item[(i)] ${\rm scn} (\mathcal{N}_{\ell})\leq 2|E(\mathcal{N})|-n_1$, where $n_1$ is the number of pendant vertices of $\mathcal{N}$.
\item[(ii)] $\mathcal{N}_{\ell}$ has an optimal shortcut set.
\end{proposition}

\begin{proof}
The bound of statement (i) is given in Corollary \ref{bound}.
To prove statement (ii), suppose that ${\rm scn}(\mathcal{N}_{\ell})=h$, and consider the mapping $$f: \overbrace{\mathcal{N}_{\ell} \times \dots \times \mathcal{N}_{\ell}}^{{\tiny 2h)}} \longrightarrow \mathbb{R}, \hspace{0.5cm} f(p_1,q_1,p_2,q_2,\dots ,p_h,q_h)={\rm diam}(\mathcal{N}_{\ell} \cup \mathcal{S})$$ where $\mathcal{S}$ is the set of segments with endpoints $p_i, q_i$, $1\leq i\leq h$. Since this is a continuous mapping taking values on a compact set, it has a minimum. On the other hand,   $\mathcal{N}_{\ell}$  admits a shortcut set and so  the minimum of $f$ must be attained by an optimal shortcut set.
\end{proof}

\begin{remark} One can check that the {\em star} $S_n$ with odd $n$, given by $V(S_n)=\{u,u_0,...,u_{n-1}\}$,  $E(S_n)=\{uu_i:0\leq i\leq n-1\}$, and vertex $u_i$ placed at
$(\cos(\frac{2\pi}{n}i),\sin(\frac{2\pi}{n}i))$ attains the upper bound of the preceding proposition.\end{remark}


Regarding the shortcut number, it is interesting to consider non-connected networks $\mathcal{N}$; note that all our previous results on this paper deal with connected networks. Obviously, when $\mathcal{N}$ is non-connected, ${\rm diam}(\mathcal{N}_{\ell})$ is infinite and a shortcut set is simply  a set of segments that connects $\mathcal{N}_{\ell}$. Thus, the following proposition is straightforward.

\begin{proposition}
Let $\mathcal{N}$ be a non-connected plane Euclidean network, and let $\mathcal{N}_{\ell}^1, \ldots , \mathcal{N}_{\ell}^k$ be the connected components of its locus. Then, ${\rm scn} (\mathcal{N}_{\ell})=1$ if and only if the set  $\{ CH(\mathcal{N}_{\ell}^1), \ldots , CH(\mathcal{N}_{\ell}^k) \}$ admits a stabbing line, i.e., a straight line that intersects each of the $k$ given convex hulls $CH(\mathcal{N}_{\ell}^i)$.
\end{proposition}

 Atallah and Bajaj \cite{AB87} designed an algorithm for line stabbing simple objects in the plane which, for our set of $k$ convex hulls,  runs in $O(k \log k)$ time. Hence, this is the complexity time of deciding  whether ${\rm scn} (\mathcal{N}_{\ell})=1$ for non-connected plane Euclidean networks $\mathcal{N}$. Nevertheless, the computation of the shortcut number in general is a much more complex computational problem. Indeed, consider the following problem:

\,

\,

\noindent {\sc \small{Short-Cut-Number:}}
 \emph{Given a plane Euclidean network (not necessarily connected) $\mathcal{N}$ with $n$ vertices and a natural number $t$, to decide whether ${\rm scn} (\mathcal{N}_{\ell})\leq t$ or not.}

\,

\begin{theorem}
{\sc \small{Short-Cut-Number}} is an NP-complete problem.
\end{theorem}
\begin{proof}
Clearly {\sc Short-Cut-Number} is in NP since, by Lemma \ref{lemma:algorithm-diameter}, one can check in polynomial time whether $|{\cal S}|\leq t$ and ${\rm diam}(\mathcal{N}_{\ell}\cup {\cal S})\leq t$ for a given set of segments ${\cal S}$ and a network $\mathcal{N}$.
 On the other hand, we can reduce 3--satisfiability (3SAT) to {\sc Short-Cut-Number} by mimicking the reduction given in \cite{MT82} of 3--satisfiability to the Point Covering Problem. For the sake of completeness, we summarize here the main steps of that reduction.

 The input of the Point Covering Problem is a set of points in the plane with rational coordinates, and one has to find a collection of straight lines of minimum cardinality such that each point lies on at least one of the lines.

Let $\phi $ be an instance of 3SAT with $m$ clauses and $n$ variables. Without loss of generality, it may be assumed that the bipartite graph of variables and clauses associated to $\phi$ is connected.  The following properties establish the reduction of \cite{MT82}.
\begin{enumerate}
\item Each clause $c_j$ is represented by a point $P_j$.
\item Each variable $x_i$ is represented by a grid of $m^2$ points $P^i_{kl}$ ($1 \leq k,l \leq m$).
\item For each $i$ ($i=1, \ldots , n$) and $j$ ($j=1, \ldots , m$), the points $P^i_{1j}, \ldots ,P^i_{mj}$ lie on a straight line denoted by $L_{ij}$, and  the points $P^i_{j1}, \ldots ,P^i_{jm}$ lie on a straight line denoted by $R_{ij}$.
\item Except for the lines defined in the item above, no other straight line of the plane contains more than two points of the point sets defined in $(1)$ and $(2)$.
\item For every $j$ ($j=1, \ldots , m$), the point $P_j$ lies on the line $L_{ik}$ if and only if $j=k$ and the positive literal $x_j \in c_j$ and  $P_j$ lies on the line $R_{ik}$ if and only if $j=k$ and the negative literal $\bar{x_j} \in c_j$.
\end{enumerate}
In our case the whole collection of points given in $(1)$ and $(2)$ is a plane Euclidean network without edges $\mathcal{N}_{\ell}(\phi)$ whose
diameter is infinite. Thus, the shortcut number of $\mathcal{N}_{\ell}(\phi)$ is the minimum number of segments that have to be inserted to $\mathcal{N}_{\ell}(\phi)$ to connect it in a unique connected component; this gives a finite diameter.

 A unique connected component can be obtained if and only if the entire collection of vertices of $\mathcal{N}_{\ell}(\phi)$ can be covered. As in \cite{MT82}, we can see that $\phi$ is satisfiable if and only if the entire collection of points given in $(1)$ and $(2)$ above can be covered by $nm$ lines.
\end{proof}

\subsection{Polygons}

In our setting, a \emph{polygon} $P$ is the locus of a plane Euclidean cycle.

\begin{proposition}\label{prop:polygon}
No polygon $P$ admits a simple shortcut, and moreover:
\begin{enumerate}
\item[(i)]  ${\rm scn} (P)=2$ when $P$ is convex.
\item[(ii)] ${\rm scn} (P)=1$ when $P$ is non-convex. In this case, there always exists a shortcut
with at least one vertex of $P$ as endpoint.
\end{enumerate}
\end{proposition}
\begin{proof}
If the two endpoints of a segment $s$ are the only intersection points with a polygon $P$, then $s$ splits $P$ into two paths whose midpoints are diametral points of $P$ and also of $P\cup s$. This implies that ${\rm diam}(P)={\rm diam}(P\cup s)$. Hence $P$ has no simple shortcut, and moreover, ${\rm scn} (P)\geq 2$ when $P$ is convex. In this case, we take two consecutive vertices of $P$ and add two segments $s$ and $s'$, each of which is very close to one of the vertices (this means that its length can be taken as small as desired) and its two endpoints are, respectively, on the two incident edges with the corresponding vertex. Clearly, either $s$ or $s'$ shortens at least one of the two paths connecting any two diametral vertices of $P$ (recall that ${\rm diam}(P)$ is attained by two vertices when $P$ is convex). Further, by taking the length of those segments small enough,  ecc$(p)<{\rm diam}(P)$ for $p\in s\cup s'$. Therefore, ${\rm diam}(P\cup \{s,s'\})< {\rm diam}(P)$, which proves statement (i).

 When $P$ is non-convex, there are two vertices $u,v\in V(P)$ such that $uv$ is not an edge of $P$ but it is a segment on the border of its convex hull $CH(P)$. The $u$-$v$ path contained in the interior of $CH(P)$ is called the {\em $u$-$v$ pocket} of $P$. Consider a segment $s$ that has vertex $u$ as an endpoint and intersects the $u$-$v$ pocket, say in a point $r'$ (see Figure \ref{fig:proofL3}). We can assume $r'$ to be very close to $v$, i.e., the length of the segment $r'v$ is some $\varepsilon>0$; the other endpoint of $s$, denoted by $r$, is located outside the pocket. As Figure~\ref{fig:proofL3} shows, given diametral points $p,q\in P$, segment $s$ shortens one of the $p$-$q$ paths in $P$; the figure distinguishes three cases according to the possible positions of $p$ and $q$.  Moreover, by taking $\varepsilon$ small enough, ecc$(z)<{\rm diam}(P)$ for every $z\in s$. Hence, $s$ is a shortcut  and so statement (ii) follows.\end{proof}

\begin{figure}[ht]
\begin{center}
\begin{tabular}{ccccccc}
  \includegraphics[width=4cm]{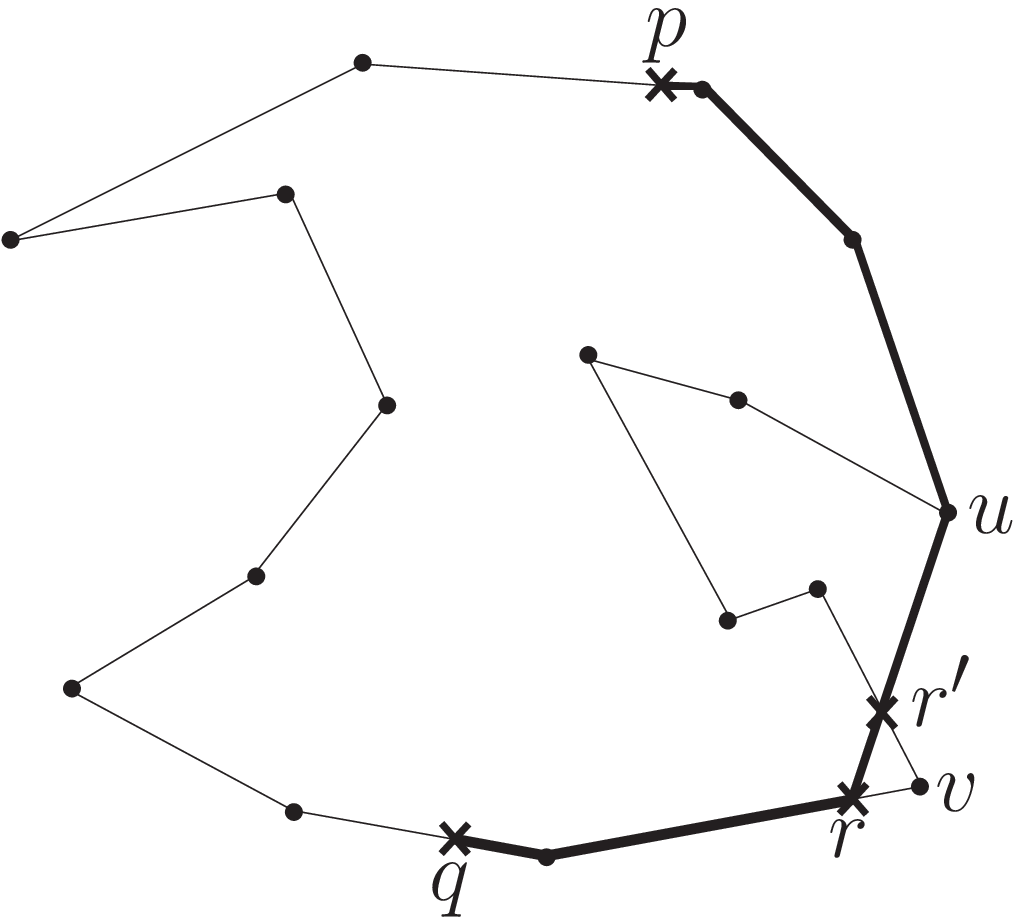}&&&
  \includegraphics[width=4cm]{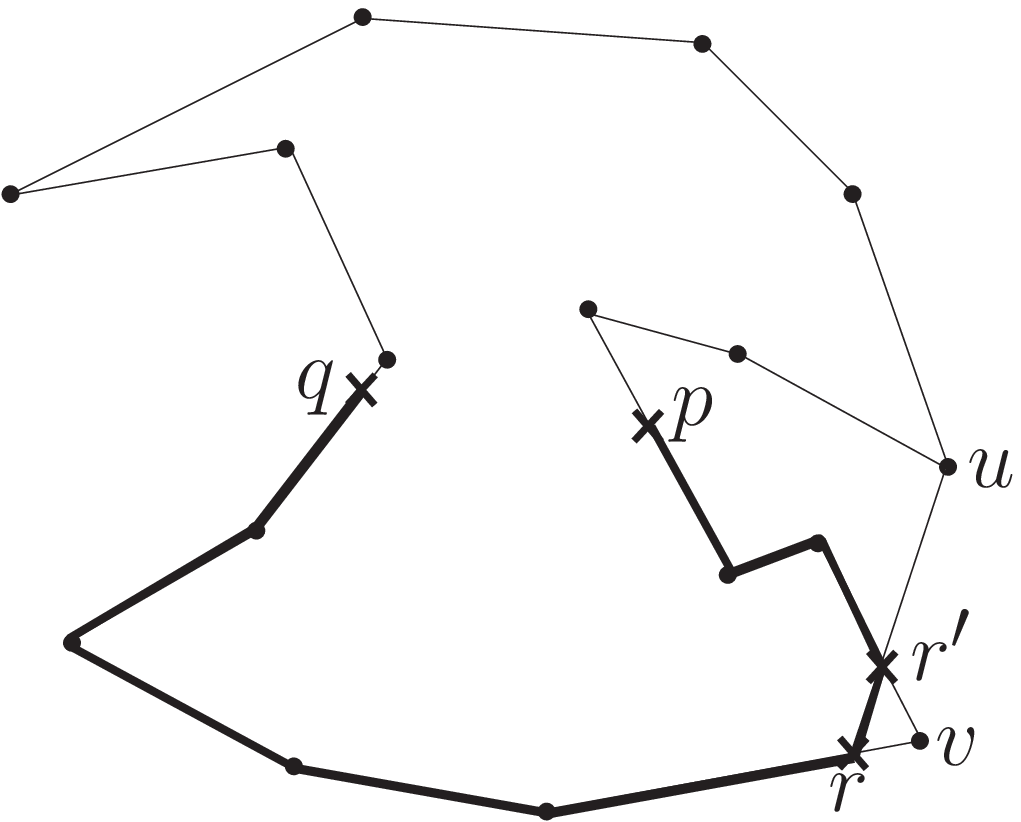}&&&
  \includegraphics[width=4cm]{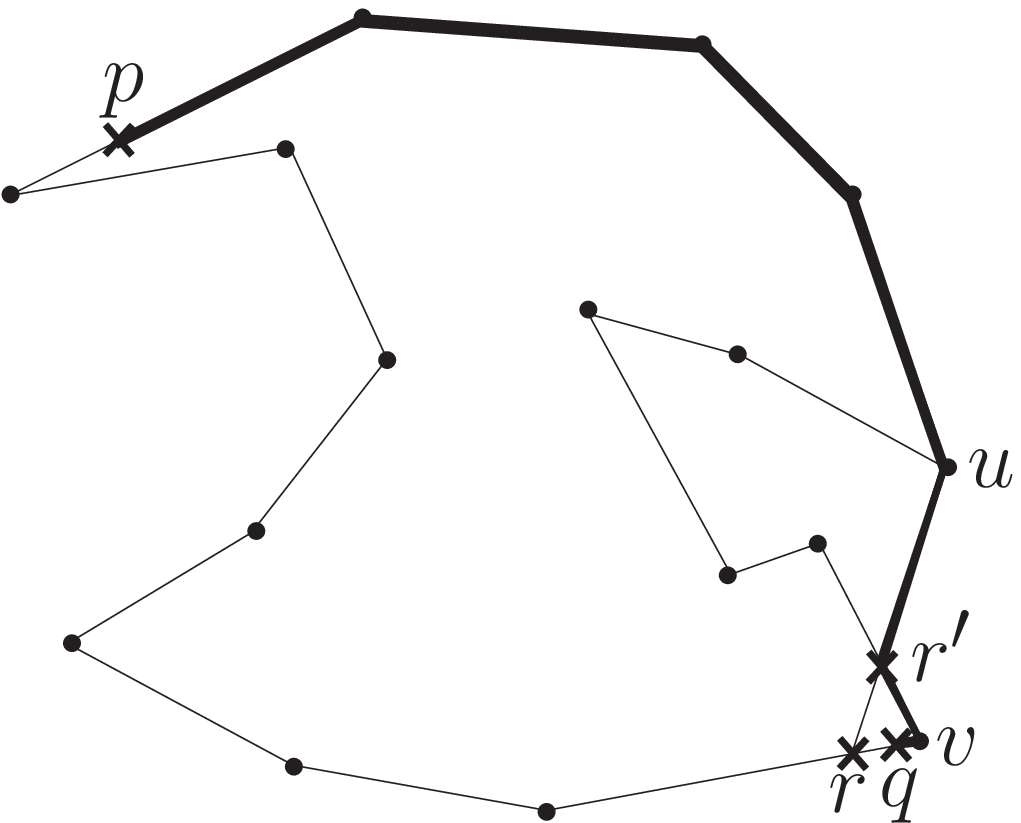}\\
(a)&&&(b)&&&(c)
\end{tabular}
\end{center}
\caption{Segment $s$ shortens one of the $p$-$q$ paths in $P$: (a) $p$ and $q$ are on the $u$-$r$ path $P_1$ that does not pass through $v$, (b) $q\in P_1$ and $p$ is on the $u$-$r'$ path that does not go through $v$, (c) one of the points is on the wedge determined by $r, v$ and $r'$.
}\label{fig:proofL3}
\end{figure}

Not all triangulations of polygons admit a shortcut; Figure~\ref{fig:proofL3yK4} illustrates an example that requires two segments to decrease the diameter. However, it seems that most polygon triangulations have shortcut number equal to one. One can check this, for example, for fan triangulations of convex polygons (it suffices to add a sufficiently small segment that intersects all incident edges with the apex of the triangulation).

\begin{figure}[ht]
\begin{center}
\begin{tabular}{c}
  \includegraphics[width=3cm]{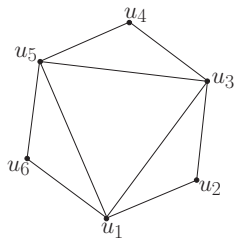}
\end{tabular}
\end{center}
\caption{Polygon triangulation with shortcut number  2 (the pairs $\{u_1,u_4\}$, $\{u_2,u_5\}$ and $\{u_3,u_6\}$ are diametral).
}\label{fig:proofL3yK4}
\end{figure}

\begin{remark}
 As it was explained in the Introduction, De Carufel et al.~\cite{DCGMS15} consider shortcuts for cycles that are always simple. Thus, their Lemma 3.1 says that no polygon admits a simple shortcut as we also prove in Proposition \ref{prop:polygon}. Since our definition of shortcut set differs from their shortcuts, we have included Proposition \ref{prop:polygon} to show how to deal with non-simple shortcuts, which is the content of statement (ii). The other statements are included for completeness.
\end{remark}

\subsection{The complete graph $K_4$}

Non-trivial Euclidean networks obtained from planar embeddings of complete graphs $K_n$ only appear for $n=3$ and $n=4$. As an abuse of notation, let $K_3$ and $K_4$ denote those Euclidean networks, and $(K_3)_{\ell}$ and $(K_4)_{\ell}$ their locus. For $n=3$,
Proposition~\ref{prop:polygon} yields ${\rm scn}((K_3)_{\ell})=2$.

\begin{proposition}\label{prop:K4} The shortcut number of $(K_4)_{\ell}$ is equal to 1.
\end{proposition}

\begin{proof}
Let $V(K_4)=\{u_1,u_2,u_3,u_4\}$ where $u_4$ lies in the interior of the triangle  $T$ of vertices $u_1,u_2,u_3$. Let per$(T)$ denote the perimeter of $T$. It follows that ${\rm diam}((K_4)_{\ell})\geq {\rm per}(T)/2$ since there is always a point $w\in u_2u_3$ such that $d(u_1,w)={\rm per}(T)/2$. Also, it holds that $d(u_1,u_4)+d(u_4,u_2) > d(u_1,u_2)$ and $d(u_1,u_4)+d(u_4,u_3) > d(u_1,u_3)$, and so there is a point $z\in u_1u_4$ sufficiently close to $u_1$ satisfying that $d(z,w)=d(z,u_1)+d(u_1,w)$. This implies that ${\rm diam}((K_4)_{\ell})> {\rm per}(T)/2$; otherwise one would have $d(z,w)=d(z,u_1)+d(u_1,w)=d(z,u_1) + {\rm diam}((K_4)_{\ell})$.
Therefore, given two diametral points $p,q$ in $(K_4)_{\ell}$, we can assume $p\in T$ and $q\notin T$.  Without loss of generality, suppose that $p\in u_2u_3$ and $q\in u_1u_4$ (see Figure~\ref{fig:proofL3yK4(2)}).

By Lemma~\ref{lemma:2shortest}, there exist at least two shortest $p$-$q$ paths in $(K_4)_{\ell}$, one of them containing vertex $u_3$ and the other containing vertex $u_2$.
 Further,  one of those paths contains vertex $u_1$. Hence,  every vertex among $u_1,u_2,u_3$ belongs to a certain shortest $p$-$q$ path. Thus,  a
segment $s$ placed sufficiently close to $u_1$ (it also applies for $u_2$ and $u_3$) and intersecting all incident edges with $u_1$ is a shortcut for $(K_4)_{\ell}$. Observe that $s$ must be located close to $u_1$ to guarantee that ecc$(z)<{\rm diam}((K_4)_{\ell})$ for every $z\in s$.
\end{proof}

\begin{figure}[ht]
\begin{center}
\begin{tabular}{c}
   \includegraphics[width=4.5cm]{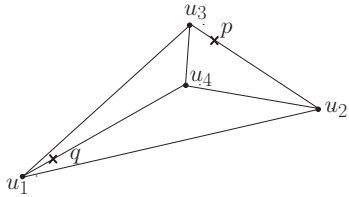}\\
\end{tabular}
\end{center}
\caption{Diametral points $p,q$ on a planar embedding of $K_4$.
}\label{fig:proofL3yK4(2)}
\end{figure}

\section{Concluding remarks}\label{sec:CCRR}

The concept of shortcut set is natural in the context of Euclidean networks, and we believe that its study is of interest not only from the theoretical point of view but also for the already mentioned applications of those networks to road networks, robotics, telecommunications networks, etc. This paper presents the first approach to its study for general plane Euclidean networks. As main results we highlight that, while the problem of minimizing the size of a shortcut set is NP-complete, plane Euclidean networks whose locus admits a shortcut can be identified in polynomial time.

Our study opens many possibilities in this type of problems, such as its extension to other network parameters or to design methods for computing optimal shortcut sets.  We leave for development in another paper to define shortcut sets for the radius and compare the results with those that we have obtained in this paper for the diameter.


\end{document}